\documentclass[12pt]{amsart} 
\usepackage{amsmath,amssymb,amsthm,kpfonts}
\usepackage{ifthen,verbatim}

\usepackage[urlcolor=blue,colorlinks=true]{hyperref}

\usepackage{floatflt,float,graphicx,graphics}
\usepackage{mathrsfs}

\usepackage{tikz}

\numberwithin{equation}{section}
\nonstopmode
\theoremstyle{plain}
\newtheorem{cor}[equation]{Corollary}

\newtheorem{theorem}[equation]{Theorem}

\newtheorem{lem}[equation]{Lemma}
\newtheorem{lemma}[equation]{Lemma}
\newtheorem{prop}[equation]{Proposition}
\newtheorem{proposition}[equation]{Proposition}

\newtheorem{thm}[equation]{Theorem}
\theoremstyle{definition}
\newtheorem{rem}[equation]{Remark}
\newtheorem{remark}[equation]{Remark}

\newtheorem{definition}[equation]{Definition}
\newtheorem{nonsec}[equation]{}

\newenvironment{pf}[1][]{%
 \vskip 3mm
 \noindent
 \ifthenelse{\equal{#1}{}}%
  {{\slshape Proof. }}%
  {{\slshape #1.} }%
 }%
{\qed\bigskip}

\newcounter{alphabet}

\newcommand{\be}{\begin{eqnarray}}
\newcommand{\ee}{\end{eqnarray}}
\newcommand{\ba}{\begin{array}}
\newcommand{\ea}{\end{array}}
\newcommand{\ben}{\begin{eqnarray*}}
\newcommand{\een}{\end{eqnarray*}}

\newcommand{\diam}{{\operatorname{d}}}
\newcommand{\card}{{\operatorname{card}\,}}

\newcommand{\R}{{\mathbb R}}

\newcommand{\bB}{\overline{B}}   
\newcommand{\B}{\mathbb{B}}

\newcommand{\capa}{\mathrm{cap}\,}

\newcommand{{\tth}}{\mathrm{th}}

\newcommand{\mb}{\mathbb}

\newcommand {\Bbn} {{\overline{B}^n}}
\newcommand {\Sn} {{\overline{\mathbb R}^n}}

\newcommand {\aand} {\quad\text{and}\quad}
\newcommand {\mM} {\mathsf{M}}
\newcommand {\M} {\mathsf{M}}
\newcommand {\UP} {{\rm UP}}
\renewcommand {\mod} {{\rm mod}}

\newcounter{minutes}
\divide\time by 60
\newcounter{hours}
\multiply\time by 60
\addtocounter{minutes}{-\time}

\begin{document}
\bibliographystyle{amsplain}
\title
{Uniformly Perfect Sets, Hausdorff Dimension, and Conformal Capacity}

\makeatletter\def\thefootnote{\@arabic\c@footnote}\makeatother

\author[O. Rainio]{Oona Rainio}
\address{Department of Mathematics and Statistics, 
University of Turku, FI-20014 Turku, Finland}
\email{ormrai@utu.fi
 \newline ORCID ID: \url{http://orcid.org/0000-0002-7775-7656}}
\author[T. Sugawa]{Toshiyuki Sugawa}
\address{Graduate School of Information Sciences,
Tohoku University, Aoba-ku, Sendai 980-8579, Japan}
\email{sugawa@math.is.tohoku.ac.jp  \newline ORCID ID:\url{https://orcid.org/0000-0002-3429-5498}}
\author[M. Vuorinen]{Matti Vuorinen}
\address{Department of Mathematics and Statistics, University of Turku,
 FI-20014 Turku, Finland}
\email{vuorinen@utu.fi \newline ORCID ID: 
 \url{http://orcid.org/0000-0002-1734-8228}}

\keywords{Condenser capacity, invariant metrics, modulus of a curve family, uniformly perfect set, Whitney cubes.}
\subjclass[2010]{Primary 30F45; Secondary 30C85}
\begin{abstract} Using the definition of uniformly perfect sets in terms of convergent sequences, we apply lower bounds for the Hausdorff content of a uniformly perfect subset $E$ of $\mathbb{R}^n$ to prove  new explicit lower bounds for the Hausdorff dimension of $E.$ These results also yield lower bounds for capacity test functions, which we introduce, and enable us to characterize domains of $\R^n\,$ with uniformly perfect boundaries. Moreover, we show that an alternative method to define capacity test functions can be based on the Whitney decomposition of the domain considered.
\end{abstract}

\thanks{Funding. 
The research of the first author was funded by Magnus Ehrnrooth Foundation. 
The research of the second and the third authors was supported in part by JSPS KAKENHI Grant Number JP17H02847.
}
\maketitle
\medskip
\hspace{8cm} {In memoriam: Pentti J\"arvi 1942- 2021}

\medskip

\medskip

\section{Introduction}
Conformal invariants, in particular the modulus 
of a curve family and the conformal capacity of a condenser, 
are fundamental tools of geometric function theory and 
quasiconformal mappings \cite{du,GM,GMP,HKV20,V71}. 
For applications, many upper and lower bounds for conformal invariants 
have been derived 
in terms of various geometric functionals. All this research 
shows that the metric structure of the boundary has a strong 
influence on the intrinsic geometry of the domain 
of the mappings studied. Indeed, many results originally proven 
for functions defined in the unit ball $\B^n $ of $\R^n, n\ge 2,$ 
can be extended to 
the case of subdomains $G \subset \R^n$ if the boundary $\partial G$ is 
"thick enough" in the sense of capacity, or, more precisely, if the 
boundary is uniformly perfect. The thickness of the boundary has a strong
influence on the intrinsic geometry of the domain and thus it also gives
a restriction on the oscillation of a function defined in $G.$
We give several new characterizations 
of uniformly perfect sets.

A \emph{condenser} is a pair $(G,E)$ where $G\subset\R^n$, $n\geq2$, 
is a domain and   $E\subset G$ is a compact set  \cite{du,GMP, HKV20}. 
A compact set $E\subset\R^n$ is of conformal capacity zero if, for 
some closed ball $B\subset \R^n \setminus E\,, $ the condenser 
$(\R^n \setminus B,E)$ has  capacity  zero, written as $\capa(E)=0$ 
with notations of Definition \ref{def_condensercapacity}. Sets of 
capacity zero are very thin, their Hausdorff dimensions are zero \cite[p.120, Cor. 2]{R2},
\cite[Lemma 9.11]{HKV20}, and they often have the role of 
negligible exceptional sets in potential theory or geometric 
function theory. Note that, due to the M\"obius invariance of 
the conformal capacity, the notions of zero and positive capacity 
immediately extend to compact subsets of the M\"obius space 
$\overline{\R}^n=\R^n\cup\{\infty\}$.

Here our goal is to study those subsets of $\R^n$ that have a positive 
capacity instead. However, the structure of sets of positive capacity 
can sometimes be highly dichotomic, for instance, in the case of 
$E_1\cup E_2$ where $E_1\subset\R^2$ is a segment and $E_2$ is a point 
not contained in $E_1$. This kind of a dichotomy makes working 
with these sets difficult, but  a subclass of sets with positive capacity, 
\emph{uniformly perfect sets}, has certain natural properties useful for
our purposes.
During the past two decades, uniformly perfect sets have become ubiquitous 
for instance in geometric function theory \cite{AW,GM}, analysis on
metric spaces \cite{H,Lew}, hyperbolic geometry \cite{BM,KL} and in the 
study of complex
dynamics and Kleinian groups \cite{BJ, SS}.

We begin by giving a variant of the definition of uniformly perfect
sets in terms of convergent sequences as follows.
For $0<c<1,$ let $\UP_n(c)$ denote the
collection of compact sets $E$ in $\mathbb{R}^n$ with 
$\card(E)\ge 2$ satisfying the condition
\begin{equation*}
\{x\in E: c r<|x-a|<r\}\ne\emptyset \quad\text{for all}~ a\in E
\aand 0<r< d(E)/2\,.
\end{equation*}
We say that a set is uniformly perfect if it is in the class $\UP_n(c)$
for some $c \in (0,1).$

The first lemma has an important role in the sequel.

\begin{lem} ({\rm T. Sugawa}  \cite[Proposition 7.4]{Sugawa98})\label{introHDimBd}
Let $E\in\UP_n(c)$ for some $0<c<1.$
Then for every $a\in E,~ a\ne\infty,$ the Hausdorff content of $E$
has the following lower bound
$$
\Lambda^\beta(E\cap\Bbn(a,r))\ge \frac{r^\beta}{2\cdot 3^n}, \quad 0<r<d(E)/2,
~\text{where}~ \beta=\frac{\log 2}{\log(3/c)}.
$$
Moreover,  the Hausdorff dimension $\dim_H(E)$ of
$E$ is at least $(\log 2)/\log(3/c).$
\end{lem}

The explicit bounds we obtain in this paper depend on the above result
of  Sugawa, however, we prove it  here with the above refined form of 
the constant $\beta.$
Moreover, we also apply ideas from the work of 
 Reshetnyak \cite{R1,R2} and   Martio \cite{M}, see also Remark \ref{Cartan},
but now the constants are explicit which is crucial for what follows. 
In the study of uniformly perfect sets, similar methods were also applied 
by  J\"arvi and Vuorinen \cite[Thm 4.1, p. 522]{JV96}.
 Our results here yield explicit constants for several characterizations of
 uniform perfectness such as the following main result. 
 
\begin{thm}\label{introCapBd}
Let $E\in\UP_n(c)$ for some $0<c<1.$
Then for every $a\in E,~ a\ne\infty,$ and all $r \in (0, d(E))$ 
the following lower bound for the conformal capacity ${\rm cap}(a,E,r)$
of the condenser $(B^n(a,2r), E \cap \overline{B}^n(a,r))$ holds
$$
{\rm cap}(a,E,r) \ge \frac{1}{2\cdot 3^n M_1(n,\beta)},\quad \beta=  
\frac{\log 2}{\log(3/c)}\,,
$$
where $M_1(n, \beta),$ given in \eqref{eq:M1b}, 
is an explicit constant depending only on $n$ and $\beta.$
\end{thm} 

There are many equivalent definitions of uniformly perfect sets. The definition given
in \cite{JV96} says that a set$E$ is $\alpha$-uniformly perfect, if the moduli of the
ring domains separating the set have the upper bound $\alpha.$ We show in Section 6 that
this definition is quantitatively equivalent with the definition of $\UP_n(c).$ Moreover,
we prove this equivalence with explicite constants, a fact which would enable one to
give explicit constants for instance in some earlier results such as in 
\cite[Thm 4.1]{JV96}.

Suppose now that $G\subset \R^n$ is a domain, its boundary $\partial G$ is of positive capacity, and define
$u_{\alpha}: G \to (0,\infty)$ by
\[u_{\alpha}(z) = \capa (G, \overline{B}^n(z, \alpha d(z,\partial G)))\,, \quad \alpha \in (0,1)\,,
\]
for $z \in G$ where $d(z,\partial G)$ is the distance from $z$ to $\partial G\,.$ 
We call $u_{\alpha}(z)$ \emph{the capacity test  function}  of $G$ at the point 
$z \in G\,.$
The numerical value of the capacity test  function depends clearly on $G, z,$ and $\alpha\,,$ 
but we omit $G$ from the notation because it is usually understood
from the context. Clearly, the capacity test  function is invariant under 
similarity transformations. We will also show that it is continuous as 
a function of both $z \in G$ and $\alpha \in (0,1)\,.$ For the purpose of
this paper, it is enough to choose e.g. $\alpha =1/2\,.$

Analysing the capacity test  function $u_{\alpha}(z)$ further, we show that, 
for a fixed $\alpha,$ it satisfies \emph{the Harnack inequality} as a 
function of $z$, a property which has a
number of consequences. First, for every $z_0 \in G,$ we see that 
$u_{\alpha}(z_0)>0\,,$
because the boundary was assumed to be of positive capacity.
Second, by fixing $z_0 \in G$, we see by a standard chaining argument 
\cite[p. 96 Lemma 6.23  and p. 84]{HKV20} that $u_{\alpha}(z)/u_{\alpha}(z_0)$ 
has a positive explicit minorant for a large class of domains, so called 
$\varphi$-uniform domains. 
This minorant, 
depending on $\varphi,\,\, d(z, \partial G)$ and Harnack parameters,  shows that $u_{\alpha}(z)$ cannot approach  $0$ arbirarily
fast when $z $ moves far away from $z_0$ or when $z \to \partial G\,.$ Under the stronger requirement that $\partial G$ be uniformly
perfect,  it follows that $u_{\alpha}(z)$ is bounded from below by a constant $c>0.$ These observations lead to the following new characterization of uniformly perfect sets. 

\begin{theorem} \label{newUP1} The boundary of a domain $G \subset \R^n$ 
is uniformly perfect if and only if there exists a constant $s>0$ 
such that $u_{\alpha}(x)\ge s$ for all $x \in G$.
\end{theorem}

Many characterizations are known for plane domains with uniformly perfect 
boundaries and often these characterizations are given in terms of 
hyperbolic geometry \cite[pp. 342-344]{GM}, \cite{KL}, \cite{BM}. Because
the hyperbolic geometry cannot be used in dimensions $n\ge 3,$
we use here another tool, the {\it Whitney decomposition} of a domain 
$G\subset {\R}^n,$ which has numerous applications to geometric function
theory and harmonic analysis \cite{B,GM,ST}. The Whitney decomposition 
represents $G$ as a countable union of non-overlapping cubes with 
edge lengths equal
to $2^{-k}, k \in \mathbb{Z},$ where the edge length is proportial to the distance from a
cube to the boundary of the domain  \cite{ST}.  Martio and Vuorinen \cite{MV} applied this 
decomposition to establish upper bounds for the metric size of the boundary $\partial G$ in
terms of $N_k,$ the number of cubes with edge length equal to $2^{-k}\,.$  Their method was based on imposing growth bounds for $N_k$ when $k \to \infty$ and depending
 on the growth rate, the conclusion was either an upper bound for
the Minkowski dimension of the boundary or
a sufficient condition for the boundary to be of capacity zero. In our next
main result, Theorem \ref{thmWUP}, we use Whitney decomposition "in the opposite direction".
Indeed, we employ Whitney cubes as test sets for the capacity structure of the boundary and 
obtain the following characterization of
uniform perfectness in all dimensions $n\ge 2$. Whitney cubes also have applications to
the study of surface area estimation of the level sets of the distance function \cite{KLV}.

\begin{theorem} \label{thmWUP}
The boundary of a domain $G\subset\R^n$ is uniformly perfect 
if and only if there exists a constant $s>0$ such that, for every 
Whitney cube $Q\subset G$, 
\begin{align*}
\capa(\R^n\setminus Q, \partial G)\geq s.
\end{align*}
\end{theorem}


By definition, see the property \ref{Wdecmp}(3) below, every Whitney cube 
$Q \subset G$ satisfies $d(Q)/d(Q, \partial G)> 1/4.$ Thus we see
that the next theorem generalizes Theorem \ref{thmWUP}.

\begin{theorem}\label{thm_1.3}
Let $G\subset\R^n$ be a domain and $E\subset G$ a compact set.
If $E$ and $\partial G$ are uniformly perfect, then
\begin{align*}
\capa(G,E)\geq s \log(1+d(E)\slash d(E,\partial G)),    
\end{align*}
where  $s>0$ is a constant depending only on the dimension $n$ and the uniform perfectness parameters of $E$ and $\partial G$. 
\end{theorem}

This theorem is well-known if both $E$ and $\partial G$ are continua, see \cite[Lemma 7.38, Notes 7.60]{cgqm}.

\noindent
\textbf{Acknowledgements.}
The authors are grateful to Prof. Akihiro Munemasa for
information about the kissing numbers of combinatorial geometry.
We are indebted to Prof. Don Marshall for a permission to use
his software for plotting Whitney cubes.
The second author would like to express his thanks to the
Department of Mathematics and Statistics, University of Turku,
for its hospitality and support during the visit to Turku, Finland in 2022.

\section{Preliminary results}\label{sec_metrics}

In this preliminary section we recall some basic facts about  metrics
and quasiconformal homeomorphisms. Moreover, we prove a few propositions
which are results of technical character, essential for the proofs
of the main theorems in subsequent sections.

The following notations will be used: The Euclidean diameter of the non-empty set 
$J$ is $d(J)=\sup\{|x-y|\text{ }|\text{ }x,y\in J\}$. The Euclidean distance 
between two non-empty sets $J,K\subset\R^n$ is 
$d(J,K)=\inf\{|x-y|\text{ }|\text{ }x\in J,y\in K\}$ and the distance from 
a point $x$ to the set $J$ is $d(x,J)=d(\{x\},J)$. Thus, for all points $x$ in 
a domain $G\subsetneq\R^n$, the Euclidean distance from $x$ to the boundary 
$\partial G$ is denoted by $d(x,\partial G),$ the Euclidean open ball with 
a center $x$ and a radius $r$ by $B^n(x,r)=\{y\in\R^n\text{}:\text{}|x-y|<r\}$, 
the corresponding closed ball by $\overline{B}^n(x,r)$ and their boundary by 
$S^{n-1}(x,r)=\partial B^n(x,r)$. If the center $x$ or the radius $r$ are not 
otherwise specified, assume that $x=0$ and $r=1$.  The unit ball is denoted by 
$\B^n\,.$ We denote by $\Omega_n$ the volume of the unit ball $\B^n$ and by 
$\omega_{n-1}$ the area of the unit sphere $S^{n-1}.$
As well known, $\omega_{n-1}=n\Omega_n=2\pi^{n/2}/\Gamma(n/2).$

  For $x \in \R^n, 0<a<b,$ we use the following notation for an 
  annulus centered at $x$ 
  \begin{equation}\label{myRing}
  R(x,b,a)=  \overline{B}^n(x,b)\setminus B^n(x,a).
  \end{equation}
  
  The first proposition shows that an annulus and its translation are
  both subsets of a larger annulus. This larger annulus is a superannulus
  for both of the two smaller annuli, i.e. the smaller annuli
  are separating the two boundary components of the larger annulus (Fig. 1).
  
   \medskip
 \begin{figure}[H]
    \centering
    \includegraphics[width=9cm]{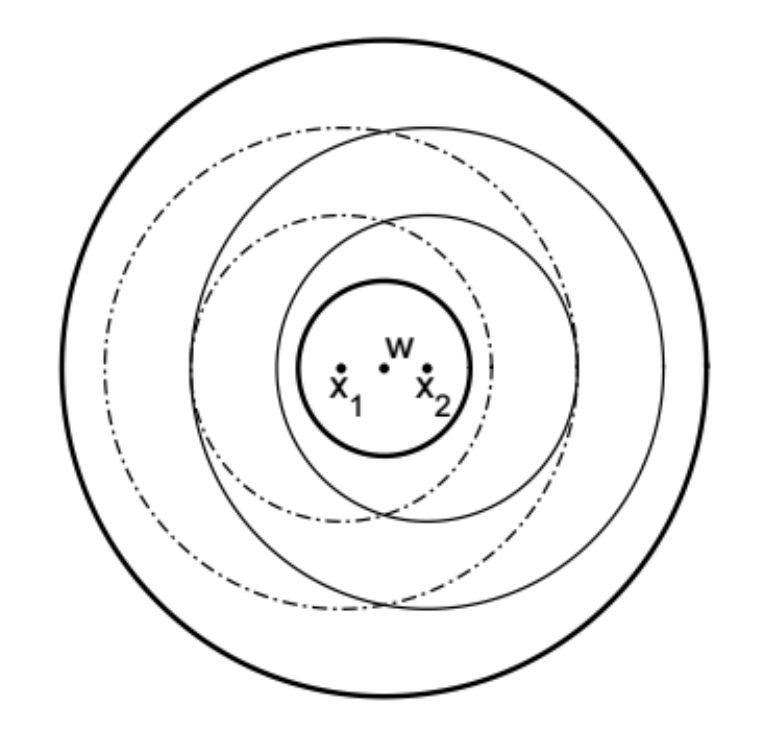}
    \caption{An annulus centered at $x_1$ and its translation  centered
    at $x_2$ (marked with  solid
    and dash-dot markers, resp.) are both subsets of a larger annulus (marked with thick marker) 
    centered at the midpoint $w=(x_1+x_2)/2.$ The larger annulus is a common superannulus
    of the two smaller annuli. }\label{Fig1}
  \end{figure}

  \begin{prop}\label{ringProp} Let $ 0<a<b$ and $\tau \ge 2.$
 If $x_1,x_2 \in \R^n, |x_1-x_2|<a,$  and $w=(x_1+x_2)/2$, then for $j=1,2$
 $$
 R(x_j,b,a)   \subset R(w,d, c), \quad \frac{d}{c} \le 
 \frac{b}{a} \frac{1+ |x_1-x_2|/b }{1-|x_1-x_2|/a }\,,
  $$
where $c= a- |x_1-x_2|, d= b+ |x_1-x_2|\,. $  In particular, if
 $|x_1-x_2|< a/\tau^2$, then for $j=1,2$
  $$
  R(x_j,\tau b, a/\tau) \subset R(w,\tau^2 b, a/\tau^2) .
  $$
  Moreover,  $R(w,\tau^2 b, a/\tau^2) \subset B^n(x_j,2\tau^2 b),$ for $j=1,2.$
   \end{prop}
   
   \begin{proof} The claims follow from the triangle inequality
and we prove here only the second one. Without loss of generality we may assume that $j=1.$ Fix
   $u \in   R(x_1,\tau b, a/\tau).$ Then $\tau b > |u-x_1|> a/\tau$ and
   $$
   |u-w|\ge | u-x_1+ x_1-w|\ge | u-x_1|- |x_1-w|\ge a/\tau- a/\tau^2 \ge a/\tau^2
   $$
   where the last inequality holds because $\tau \ge 2.$ Similarly,
   $$
   |u-w|\le | u-x_1+ x_1-w|\le | u-x_1|+ |x_1-w|\le \tau b+a/\tau^2 \le \tau^2 b
   $$
   because $\tau \ge 2.$ Therefore 
   $a/\tau^2 \le |u-w|\le \tau^2 b$ and hence $u\in R(w,\tau^2 b, a/\tau^2).$
   The last assertion also follows easily from the triangle inequality.
   \end{proof}

Topology in $\R^n$ or in its one point compactification 
$\overline{\R}^n = \R^n \cup \{\infty\}$ is the metric topology defined 
by the chordal metric.
The \emph{chordal (spherical) metric} is the function 
$q:\overline{\R}^n\times\overline{\R}^n\to[0,1],$
\begin{align*}
q(x,y)&=\frac{|x-y|}{\sqrt{1+|x|^2}\sqrt{1+|y|^2}},\text{ if }x\neq\infty\neq y,\quad
q(x,\infty)=\frac{1}{\sqrt{1+|x|^2}}. 
\end{align*}
Thus, for instance,  an unbounded domain $G\subset \R^n$ has $\infty$ as one of
its boundary points.

For the  sake of convenient reference we record the following simple inequality.

\begin{proposition} \label{quadProp}
For $a>0$ and $ p\ge 1$ the inequality 
\[ t-a\ge \lambda^p \sqrt{1+t}, \quad \lambda > 1,\]
holds for all $t \ge  \lambda^{2p} +  \lambda^p  \sqrt{1+a} +a.  $ 
\end{proposition}

\begin{proof} The inequality  clearly holds for all large values of $t,$ 
larger than $|x|,$
where $x$ is a root of
$(x-a)^2 -  \lambda^{2p} (x +1)=0. $
By the quadratic formula we see that both roots have absolute value
less than $ \lambda^{2p} +  \lambda^p  \sqrt{1+a} +a.$
\end{proof}

For two points in a domain, the next result gives an estimate for the number
of annular domains separating these points. This result has an important role
in the proof of one of our main results in Section 7.

\begin{proposition}\label{prop_rhoE4bound}
 Let $G\subset\R^n$ be a domain, let $E$ be  a compact subset of $G\,,$ 
and choose $x,y\in E$ such that
 $d(x,\partial G)= d(E,\partial G)$ and $|x-y|\ge d(E)/2$.
  Fix $z_0\in\partial G$ so that $|x-z_0|=d(x,\partial G)$ and 
let $w= (x+z_0)/2\,, \lambda>1$. If $d(E)/d(E,\partial G)>\lambda^p+2$ for some
integer $p \ge 1,$
 then there are $p$ disjoint annular domains
\begin{align*}
B^n(w,\lambda^md(x,\partial G)/2)\setminus 
\overline{B}^n(w,\lambda^{m-1}d(x,\partial G)/2),\quad m=1,...,p    
\end{align*}
separating $x$ and $y$. 
Moreover, for 
$d(E)/d(E,\partial G)>\max \{\lambda^p+2, \lambda^2+ 2\lambda+2\},$ we have
$$p > c \log(1+ d(E)/d(E,\partial G)), \quad c= 1/(2 \log \lambda).$$
\end{proposition}

\begin{proof} Because $\overline{B}^n(w,\lambda^p d(x,\partial G)/2)$ contains the
largest annulus, it is enough to show that $|w-y|>\lambda^p d(x,\partial G)/2$.
 By the triangle inequality,
\begin{align*}
|y-w|\geq|x-y|-|x-w|
\quad\Leftrightarrow\quad
\frac{|y-w|}{|x-w|}\geq\frac{|x-y|}{|x-w|}-1\,.
\end{align*}
Further,
\begin{align*}
|y-w|\geq\left(\frac{|x-y|}{|x-w|}-1\right)|x-w|
\ge (\frac{d(E)}{d(E,\partial G)}-1)d(x,\partial G)/2\ge
 (\lambda^p +1)d(x,\partial G)/2 \,
\end{align*}
as desired. To prove the second claim, the lower bound for $p,$ 
fix an integer $p\ge 1$ with
$$ \lambda^{p +1}+2  \ge u > \lambda^{p }+2, \quad u=  d(E)/d(E,\partial G).$$
Then 
\[
 p+1 \ge  \frac{ \log(u-2)}{ \log  \lambda}
\]
holds. To find a lower bound for $p,$ we observe that the inequality
\begin{equation} \label{pBound}
 p+1 \ge  \frac{ \log(u-2)}{ \log  \lambda}> 
 \frac{  \log(1+u)}{2 \log  \lambda}+1 
\end{equation}
holds iff
$$
\left( \frac{u-2}{ \lambda}\right)^2>1+u  
\Leftrightarrow u^2 -(4+ \lambda^2)u + 4 - \lambda^2>0.
$$ 
By Proposition  \ref{quadProp}, this holds
for all $u>  \lambda^2 +  2\lambda+2.$ 
By \eqref{pBound},  we see that this yields also the desired
lower bound for $p.$
\end{proof}

\medskip

\begin{nonsec}{\bf Quasiconformal maps and moduli of curve families.} Quasiconformal homeomorphisms
$f: G \to G'=f(G)$ between domains $G,G' \subset \R^n$ are commonly defined
in terms of moduli of curve families.
For the basic properties of the modulus $\M(\Gamma)$ of a curve family $\Gamma$, 
the reader is referred to \cite{Ah,H,GMP}, 
\cite[6.1, p. 16]{V71}, \cite{HKV20}.
According to V\"ais\"al\"a's book \cite{V71}, $K$-quasiconformal maps are
characterized by the inequality
\[
  \M(f \Gamma)/K \le \M(\Gamma) \le  K \M(f \Gamma)
\]
for every family of curves $\Gamma$ in $G$ where $f(\Gamma)= \{ f(\gamma): \gamma \in\Gamma\}\,.$

\end{nonsec}

The following monotonicity of the moduli of curve families is quite useful 
in various estimates. Let  $\Gamma_1,\Gamma_2$ be two curve families in $\Sn$.
We say that $\Gamma_2$ is minorized by $\Gamma_1$ and write
$\Gamma_2>\Gamma_1$ for it, if
every curve $\gamma\in\Gamma_2$ has a subcurve belonging to $\Gamma_1$.
For instance, $\Gamma_2>\Gamma_1$ if $\Gamma_2\subset \Gamma_1.$

\begin{lemma}[$\text{\cite[\S 6]{V71}}$]\label{cgqm_5.3}
\begin{enumerate}
\item
If $\Gamma_2>\Gamma_1,$ then $\M(\Gamma_2)\le\M(\Gamma_1)$.
\item
For a collection of curve families $\Gamma_i~(i=1,2,\dots, N),$
$$
\mM\left(\bigcup_j \Gamma_i\right)
\le \sum_{i=1}^N \M(\Gamma_i).
$$
Moreover, equality holds if the curve families are separate.
\end{enumerate}
\end{lemma}

Here the families $\Gamma_i$ are said to be \emph{separate}
if they are contained in pairwise disjoint Borel sets $E_i$
(see also \cite[\S 4.2.2]{GMP}).

Let $G$ be a domain in $\Sn$ and $E, F\subset\overline G.$
In what follows, $\Delta(E,F;G)$ will stand for the family of all the curves that
are in $G$ except for the endpoints, and 
that have one endpoint in the set $E$ and another endpoint in $F$ 
\cite[pp. 11-25]{V71}, \cite{HKV20}. 
When $G=\mathbb{R}^n$ or $G=\overline{\mathbb{R}}^n,$  we often write 
$\Delta(E,F;G)=\Delta(E, F).$

\begin{lemma}\label{cgqm_5.14}\emph{\cite[Thm 7.5, p. 22]{V71}}
If $0<a<b$ and $D=\overline{B}^n(b)\backslash B^n(a)$,
\begin{align*}
\M(\Delta(S^{n-1}(a),S^{n-1}(b);D))=\omega_{n-1}(\log({b}/{a}))^{1-n},
\end{align*}
where $\omega_{n-1}$ is the $(n-1)$-dimensional surface area of the 
unit sphere   $S^{n-1}$.
\end{lemma}

\begin{cor}\label{setsInRing}  Let $0<a<b$ and 
$E, F \subset D\equiv\overline{B}^n(b)\backslash B^n(a)$,
and let $$\M(\Delta(E,F;\R^n))\ge c>0.$$ Then there exists $\lambda=\lambda(n,c)>1$ 
such that for all $t\ge \lambda$
\begin{align*}
\M(\widetilde{\Delta})\ge c/2, \quad \widetilde{\Delta}= 
\Delta(E,F;B^n(0,t b)\setminus B^n(0,a/t)).
\end{align*}
\end{cor}

\begin{proof}
By the subadditivity of the modulus, Lemma \ref{cgqm_5.3}, 
\[
\M(\Delta(E,F;\R^n))\le \M(\widetilde{\Delta}) + 
\M(\Delta(E,F;\R^n) \setminus \widetilde{\Delta})
\]
whereas by Lemma \ref{cgqm_5.14}
\[
 \M(\Delta(E,F;\R^n) \setminus \widetilde{\Delta}) \le
  2 \omega_{n-1} (\log t)^{1-n}\le c/2
\]
for all $t \ge \lambda \equiv {\rm exp} ( (4 \omega_{n-1}/c)^{1/(n-1)})>1.$
\end{proof}

The {\it Teichm\"uller ring} is a domain in $\mathbb{R}^n$ with the complementary 
components $[-e_1,0]$ and $[s e_1,\infty], s >0.$ The modulus of the 
family of all curves joining these boundary components, 
denoted by $\tau_n(s),$  is a
decreasing homeomorphism $\tau_n:(0,\infty)\to (0,\infty)$ and admits the 
following lower bound
\begin{equation}\label{teichlow}
\tau_n(s) \ge c_n \log\left(1+ \frac{2(1+\sqrt{1+s})}{s}\right) 
\ge 2 c_n \log(1+\frac{1}{\sqrt{s}})\,,\quad 
c_n=B(\frac{1}{2(n-1)},\frac{1}{2})^{1-n} \omega_{n-2},
\end{equation}
where $B(\cdot, \cdot)$ is the beta function, $ c_2= 2/\pi$ \cite[p.114]{HKV20}.
For $n=2,$ $\tau_2$ can be expressed explicitly in terms of complete elliptic
integrals \cite[p. 123]{HKV20}. 

The function $\tau_n$ often occurs as a lower bound for moduli of curve 
families like in the following lemma, based on the spherical 
symmetrization of condensers. This lemma has found many applications
because it provides, for a pair of non-degenerate continua
$E$ and $F,$ an explicit connection between the geometric quantity
$d(E,F)/\min\{d(E),d(F)\}$ and the modulus of the family of all curves joining
the continua. Also a similar upper bound holds \cite[7.42]{cgqm},
 \cite[Rmk 9.30]{HKV20}, but the upper bound will not be needed here.

\begin{lem}\label{ringcap} Let $E$ and $F$ be continua in $\mathbb{B}^n$ with
$d(E), d(F)>0.$ Then
 \flushleft{(1)} \hspace{1cm}$ \M(\Delta(E,F;\mathbb{R}^n )) \ge \tau_n(4 m^2+4 m),$
 \flushleft{(2)} \hspace{1cm}$ \M(\Delta(E,F;\mathbb{B}^n )) \ge 
 \frac{1}{2}\tau_n(4 m^2+4 m),$
\newline
where $m=d(E,F)/\min\{d(E),d(F)\}.$ 
\end{lem}

\begin{proof} If $d(E;F)=0,$ then 
$ \M(\Delta(E,F;\mathbb{B}^n ))= \M(\Delta(E,F;\mathbb{R}^n ))=\infty$
by \cite[7.22]{HKV20} (or \cite[5.33]{cgqm}).  
If $d(E;F)>0,$ the proof of(1) follows from  
\cite[Lemma 9.26]{HKV20} (or \cite[7.38]{cgqm})
 and the proof of (2) from (1) and \cite[Lemma 7.14]{HKV20} (or \cite[5.22]{cgqm}).

\end{proof}
\begin{nonsec}{\bf Quasiconformal self-homeomorphism of a domain. }
\label{qcself}
For a proper subdomain $G$ of $\R^n$ and for a fixed point $x\in G$, we define 
a homeomorphism $f:\R^n\to\R^n$ such that $f(z)=z$ if $z=x$ or 
$|x-z|\geq d(x,\partial G)$. Furthermore, for $0<{\alpha} <\beta<1$, 
the mapping $f$ fulfills 
$f(S^{n-1}(x,{\alpha} d(x,\partial G)))=S^{n-1}(x,\beta \,d(x,\partial G))$.

Let $h:\R^n\to\R^n$ be the radial map $x\mapsto|x|^{a-1}x=h(x)$, $a\neq 0$, 
defined in \cite[16.2, p. 49]{V71}. Suppose $0<{\alpha},{\beta}<1$ and we want 
to choose $a$ so that $h(S^{n-1}({\alpha}))=S^{n-1}({\beta})$. Then 
\begin{equation} \label{ValueA}
{\alpha}^a={\beta}\quad\Leftrightarrow\quad a=\frac{\log({\beta})}{\log({\alpha})}  
\end{equation}
and, as shown in \cite[16.2]{V71}, the maximal dilatation of this map is 
$K(h)=\max\{a^{n-1},a^{1-n}\}$. By definition, $h(z)=z$ for $z\in\{0\}\cup S^{n-1}$. 
Define now a function $g$ such that $g(z)=h(z)$ for all $z\in\overline{\B}^n$ but 
$g(z)=z$ whenever $|z|\geq1$. By \cite[Thm 35.1, p. 118]{V71}, the dilatation of 
the mapping $g$ is same as the one of $h$.    

Fix now $x$ in $G$ and let $r=d(x,\partial G)$. Let $f$ be a radial 
$K$-quasiconformal map defined by $f(z)=rg((z-x)\slash r)$, $z\in\R^n$.
Then $f(x)=x$, $f(z)=z$ for all $z\in G\backslash B^n(x,r)$ and 
$f(B^n(x,{\alpha}\, r))=B^n(x,{\beta}\, r)$, similarly as above.
\end{nonsec}

We summarize the above arguments as a lemma.

\begin{lemma}\label{qcMap} For a proper subdomain $G$ of $\R^n$ and 
for a fixed point $x\in G$, there
exists a quasiconformal homeomorphism $f:\R^n\to\R^n$ such that $f(G) = G,$ 
and $f(z)=z$ if $z=x$ or
 $|x-z|\geq d(x,\partial G)$. Furthermore, for $0<{\alpha} <\beta<1$, 
 \[f(S^{n-1}(x,{\alpha} d(x,\partial G)))=S^{n-1}(x,\beta \,d(x,\partial G)), \quad
 K(f)=\max\{a^{n-1},a^{1-n}\},\]
where $a$ is the number in \eqref{ValueA}. 

\end{lemma}

 \bigskip

\section{Harnack inequality and capacity test  functions}

In the dimensions $n\ge 3$, quasihyperbolic distances defined below are 
widely used as substitutes of  hyperbolic distances.

\begin{nonsec}{\bf Quasihyperbolic metric.} \label{defQHyp}
For a domain $G\subsetneq {\mb{R}^n}$, the 
{\it quasihyperbolic metric} ${k}_G$ is 
defined by \cite[p. 39]{GMP}, \cite[p.68]{HKV20}
$${k}_{G}(x,y)=
\inf_{\gamma \in \Gamma} \int_{\gamma}\frac{|dz|}{d(z, \partial G)}, 
\quad x, y \in G,$$
where $\Gamma$ is the family of all rectifiable curves in $G$ joining $x$ and $y$.
This infimum is attained when $\gamma $ is the  {\it quasihyperbolic geodesic segment}
joining $x$ and $y\,.$ The hyperbolic metric of $\B^n$ can be also defined in terms
of a similar length minimizing property, with the weight function $2/(1-|x|^2)\,.$ 
In many ways the quasihyperbolic metric is similar to the 
hyperbolic metric, see \cite[Chapter 5]{HKV20}, but unfortunately its values
are known only in a few special cases. Fortunately, some lower bounds 
can be given in terms of the
$j_G$ metric and upper bounds can be given for a large class of domains as we will now
show.

The \emph{distance ratio metric} 
is defined  in a domain $G\subsetneq\R^n$ as the function $j_G:G\times G\to[0,\infty)$,
\begin{align*}
j_G(x,y)=\log(1+\frac{|x-y|}{\min\{d(x,\partial G),d(y,\partial G)\}}).   
\end{align*}

The lower bound 
$$j_G(x,y)\le k_G(x,y)$$ 
holds for an arbitrary domain $G\subsetneq {\mb{R}^n}$ and all $x,y \in G$ 
\cite[Cor. 5.6, p.69]{HKV20}.
\end{nonsec}

For the upper bound we introduce a class of domains for which we have
a simple upper bound of the quasihyperbolic distance. This upper bound
combined with the above lower bound provide handy estimates for many applications.

\begin{nonsec}{\bf $\varphi$- uniform domains.}
We say that a domain $G$ is $\varphi$-{\it uniform} if for all $x,y \in G$ 
\[k_G(x,y)\le \varphi(|x-y|/\min\{ d(x,\partial G), d(y,\partial G)   \})\,.\]
The special case $\varphi(t) = c\log(1+t), c>1\,,$ yields the so called 
{\it uniform domains} which are ubiquitous in geometric
function theory  \cite[p.84]{HKV20}, \cite{H}. For instance, balls and 
half-spaces and their images under quasiconformal mappings of $\mathbb{R}^n$
belong to  this class of domains. It is easy to check that all convex domains 
are $\varphi$-uniform with $\varphi(t) \equiv t\,.$ The strip domain 
$ \{ z \in \mathbb{C}: 0<{\rm Im} z <1\}$
is  $\varphi$-uniform but not uniform.
\end{nonsec}

\begin{nonsec}{\bf Harnack functions \cite[p. 96]{HKV20}.} 
Let $G \subset \mathbb{R}^n$ be a domain and
let $u:G \to (0,\infty)$ be a continuous function. We say that $u$ is 
a \emph{Harnack function} with parameters $(s,C), 0 < s < 1 <C,$ 
if for every $z \in G$ and all $x \in \overline{B}^n(z, s d(z,\partial G))$
\[
u(z) \le C u(x) \,.
\]
\end{nonsec}

\medskip 
It follows easily from the definition of the quasihyperbolic metric, 
see \cite[p. 69, Lemma 5.7]{HKV20},
that the balls $\overline{B}^n(z, s d(z,\partial G))$ in the definition 
of a Harnack function
have quasihyperbolic diameters majorized by a constant depending on 
$s$ only. For given
$x, y \in G$ one can now estimate for a Harnack function $u$ the quotient 
$u(x)/u(y)$
using the quasihyperbolic distance $k_G(x,y)$ in a simple way as 
shown in \cite[pp.94-95]{HKV20}.
In fact, we fix a quasihyperbolic geodesic in $G$ joining $x$ and $y$
\cite[p. 68, Lemma 5.1]{HKV20},
 and cover it optimally, using as few  balls 
 $\overline{B}^n(z, s d(z,\partial G))$ as possible.  In this way, we
obtain the next lemma.

\medskip

\begin{lemma} \cite[Lemma 6.23, p.96]{HKV20} \label{HarBd} Let 
$u:G \to (0,\infty)$ be a Harnack function with parameters
$(s,C)\,.$ 

(1) Then
\[
u(x) \le C^{1+t} \, u(y)\,, \quad t=k_G(x,y)/d_1(s)\,,
\]
for $x,y \in G$ where $d_1(s)= 2 \log(1+s)\,.$

(2) If a domain $G\subset \B^n, 0 \in G\,,$ is $\varphi$-uniform, 
then for $d(y,\partial G)< d(0,\partial G)$
\[
u(y)/u(0)> C^{-1-\varphi(2/d(y,\partial G))/d_1(s)}\,.
\]
\end{lemma}

We next start our study of the capacity test  function
\begin{equation}\label{uFun}
u_{\alpha}(z) = \capa (G, \overline{B}^n(z, \alpha d(z,\partial G)))\,, 
\quad \alpha \in (0,1)\,.
\end{equation}
and examine its dependence on $\alpha$ and $z$ when the other argument
is fixed. It turns out that the 
dependence of the capacity test  function
on $\alpha$ is controlled by standard ring domain capacity estimates from
\ref{qcself} whereas,  as a function of $z,$  it is continuous and
satisfies a Harnack condition.

\begin{lemma}\label{lemma_moduli_uv}
Let $G\subset\R^n$, $x\in G\,,$  $0<{\alpha}<{\beta}<1\,,$ and 
$a=\log({\beta})\slash\log({\alpha})$. Then
\begin{align*}
\M(\Gamma_{\alpha})\leq\M(\Gamma_{\beta})\leq  a^{n-1}\,\M(\Gamma_{\alpha}),   
\end{align*}
where  $\Gamma_s=\Delta(S^{n-1}(x,sd(x,\partial G)),\partial G;G)$ for 
$s={\alpha},{\beta}$.
In other words,
\[
u_{\alpha}(x) \le u_{\beta}(x)\le a^{n-1} u_{\alpha}(x) \,.
\]
\end{lemma}
\begin{proof}
The first inequality follows from Lemma \ref{cgqm_5.3} and, 
by using the quasiconformal map $f$ of \ref{qcself}, we see that 
the second inequality holds. 
\end{proof}
\medskip

The above result shows that for a fixed $x \in G,$ $u_{\alpha}(x)$ is 
continuous with respect to the parameter $\alpha$
because $a\to 1$ when $\beta\to \alpha\,.$ The next result shows, 
among other things, that
for a fixed $\alpha\in (0,1),$ $u_{\alpha}(x)$ is continuous as a function of $x\,.$
This continuity follows from the domain monotonicity of the capacity 
\eqref{domainMono} and Lemma \ref{lemma_moduli_uv}.
\medskip

\begin{theorem} \label{uHarnack} 
Let $G$ be a domain in $\R^n$ with boundary of positive capacity.
Let $\alpha \in (0,1)$ and choose $s\in(0,1)$ so small that $(1+s)\alpha+s<1.$
The capacity test  function $u_\alpha$ of $G$ is continuous on $G$ and
satisfies the Harnack inequality with parameters $(s,C),$ where
$$
C=\left(\frac{\log((1+s)\alpha+s)}{\log\alpha}\right)^{n-1}.
$$
\end{theorem}

\begin{proof}
For $z \in G$, we take  $x \in \overline{B}^n(z, s\,d(z,\partial G))\,;  $ namely,
\begin{align}\label{3.9_ine_xz}
|x-z|\leq s\,d(z,\partial G).
\end{align}
It follows from the triangle inequality and the inequality \eqref{3.9_ine_xz} that 
\begin{align}
d(x,\partial G)\leq|x-z|+d(z,\partial G)\leq(1+s)d(z,\partial G)\,, \label{3.9_ine_tri1}\\ 
(1-s)d(z,\partial G)\leq d(z,\partial G)-|x-z|\leq d(x,\partial G)\,.\label{3.9_ine_tri2}
\end{align}
Note here that $\overline{B}^n(a,r_a)\subset\overline{B}^n(b,r_b)$ if and only if $|a-b|+r_a\leq r_b$. 
By using the inequalities \eqref{3.9_ine_xz} and \eqref{3.9_ine_tri1}, we have
\begin{align*}
|x-z|+\alpha d(x,\partial G)\leq
s\,d(z,\partial G)+(1+s)\alpha\,d(z,\partial G)=
((1+s)\alpha+s)\,d(z,\partial G)\,.
\end{align*}
Similarly, by the inequalities \eqref{3.9_ine_xz} and \eqref{3.9_ine_tri2},
\begin{align*}
|x-z|+\alpha d(z,\partial G))\leq
(s+\alpha)d(z,\partial G)\leq
\frac{s+\alpha}{1-s}d(x,\partial G)\,.
\end{align*}
Therefore, we have
\[ \overline{B}^n(x,\alpha\, d(x,\partial G)) \subset  
 \overline{B}^n(z,((1+s)\alpha+s)d(z,\partial G))\,,\]  
 \[ \overline{B}^n(z,\alpha \,d(z,\partial G))  \subset  
 \overline{B}^n(x,\beta\,d(x,\partial G))\,,\] 
where $\beta = (\alpha+s)/(1-s)\,.$ 
In particular, with the help of Lemma \ref{lemma_moduli_uv}, we obtain
\begin{align*}
u_{\alpha}(x)&\le  u_{(1+s)\alpha+s}(z) \le C u_{\alpha}(z)\,, \\
u_{\alpha}(z)&\le  u_{\beta}(x) \le D u_{\alpha}(x),
\end{align*}
where $D=[\log \beta/\log\alpha]^{n-1}.$
The first inequality is nothing but the required Harnack inequality.
Since $C<D,$ we obtain
$$
|u_{\alpha}(z) - u_{\alpha}(x)|\le D
=\left(\frac{\log[(\alpha+s)/(1-s)]}{\log\alpha}\right)^{n-1}.
$$
Now the continuity of $u_\alpha$ follows because $D\to 1$ as $s\to 0.$
\end{proof}

\bigskip

\bigskip

Let $G\subset \R^n$ be a domain with $\capa (\partial G)> 0, z_1 \in G$ and fix
$\alpha=1/4\,.$  Theorem  \ref{uHarnack} and  Lemma \ref{HarBd} show that, 
perhaps surprisingly, {\it  the  speed of decrease of the function 
$u_{\alpha}(z)/u_{\alpha}(z_1)$ to $0$ when $k_G(z_1,z)\to \infty$ or $z \to \partial G$
is controlled  from below by the Harnack parameters given by Theorem 
\ref{uHarnack}  and by $k_G(z_1,z)\,.$ }   

\medskip



\section{Capacity test  function}

Various capacities are widely applied in geometric function theory to 
investigate the metric size of sets \cite{GM,du}. We use here the 
conformal  capacity of condensers and prove several lemmas involving this capacity.
We begin by pointing out the connection between the condenser capacity and the modulus
of a curve family.
These lemmas, together with the superannulus Proposition \ref{ringProp},
are applied to prove Lemma \ref{TwoSets}, which will be a key tool
for the proof of a main result in Section 7.

A domain $D\subset\overline{\R}^n$ is called a \emph{ring} if its complement 
$\overline{\R}^n\backslash D$ has exactly two components $C_0$ and $C_1$. 
Sometimes, we write $D=R(C_0, C_1).$
We say that a ring $D=R(C_0,C_1)$ \emph{separates} a set $E$, if $E\subset C_0\cup C_1$ 
and if $E$ meets both of $C_0$ and $C_1.$
As in \cite[7.16, p. 120]{HKV20}, the (conformal) \emph{modulus} of a ring 
$D=R(C_0,C_1)$ is
defined by
\begin{align*}
\mod(D)
=\left(\frac{\M(\Delta(C_0,C_1))}{\omega_{n-1}}\right)^{1\slash(1-n)}  \end{align*} 
and its capacity is $\capa(D)=\M(\Delta(C_0,C_1))$.

\begin{definition}\label{def_condensercapacity}\cite[Def. 9.2, p. 150]{HKV20}
A pair $E=(A,C)$ where $A\subset\R^n$ is open and non-empty, and $C\subset A$ 
is compact and non-empty is called a \emph{condenser}. The \emph{capacity} of 
this condenser $E$ is
\begin{align*}
{\rm cap}(E)=\inf_u\int_{\R^n}|\nabla u|^n dm,
\end{align*}
where the infimum is taken over the family of all non-negative $\text{ACL}^n$ 
functions $u$ with compact support in $A$ such that $u(x)\geq1$ for $x\in C$. 
A compact set $E$ is of 
{\it capacity zero}, denoted by  $ \capa E = 0\,,$ if $\capa(A,C)= 0\,$ for 
some bounded domain
$A, C \subset A\,.$ Otherwise we denote $\capa E>0\,$ and say that $E$ is 
of positive capacity.
Note that the definition of capacity zero does not depend on the
open bounded set $A$ \cite[pp.150-153]{HKV20}.
\end{definition}

For the definition of $\text{ACL}$ and $\text{ACL}^n$ mappings, see 
\cite[Def. 26.2, p. 88; Def. 26.5 p. 89]{V71}, \cite[6.4]{GMP}. It is useful to recall 
the close connection between the modulus
of a curve family and capacity, because many properties of curve 
families yield similar properties for the capacity. 

\begin{remark} \label{GehringZiem} \cite[p.164, Thm 5.2.3]{GMP}, 
\cite[Thm 9.6, p. 152]{HKV20}
The capacity of a condenser $E=(A,C)$ can also be expressed in terms of a modulus
of a curve family as follows:
\begin{align*}
{\rm cap}(E)=\M(\Delta(C,\partial A;A))\,.
\end{align*}
\end{remark}

 \begin{lemma}\label{marsarbd}
Let $E \subset {\mathbb R}^n$ be a compact set, $x\in {\mathbb R}^n,$ 
and let $0<r<s<t.$ Then
\[
A\ge {\capa}(B^n(x,t),E \cap \overline{B}^n(x,r)) \ge 
\left(\frac{\log(s/r)}{\log(t/r)}\right)^{n-1} A, \quad
A\equiv {\capa}(B^n(x,s),E \cap \overline{B}^n(x,r)) .
\] 
 \end{lemma}
 
\begin{proof} The proof follows immediately from Remark 
\ref{GehringZiem} and Lemma \ref{qcMap}.
 \end{proof}

Let $E=(A,C)$ be a condenser and $A_1$ a domain with $A \subset A_1\,.$ 
It follows readily from
the definition of the capacity (and also from Lemma \ref{cgqm_5.3})
that the following {\it domain monotonicity property} holds
\begin{equation} \label{domainMono}
\capa(A_1,C)\le \capa(A,C)\,.
\end{equation}

 \medskip
  For   a compact set $E \subset {\mathbb R}^n$ ,  $x\in {\mathbb R}^n,$ and  
  $r>0,$ we introduce
  the notation
 \begin{equation} \label{myDelta}
  {\rm cap}(x,E,r)= \M(\Delta(x,E,r))\,, \quad
  \Delta(x,E,r)= \Delta(S^{n-1}(x,2r), 
  E \cap \overline{B}^n(x,r); {\mathbb R}^n)\,.
  \end{equation}

  \medskip
  The condition $ {\rm cap}(x,E,r)>0$ gives information about the size of the set
  $E$ in a neighborhood of the point $x.$ The next lemma shows that there is
  a substantial part of the set $E,$ in the sense of capacity, in an annulus
  centered at $x.$
  
  \begin{lemma}\label{basicLem}
Let  $E \subset {\mathbb R}^n$  be  a compact set,  $x\in {\mathbb R}^n,$ and  $r>0$ and 
suppose that
\[
 \M(\Delta(x,E,r) )= c >0.
\]
If $ \Delta_1=\Delta(S^{n-1}(x,2r), E \cap \overline{B}^n(x,r); B^n(x,2r)),$ then
\[
 \M(\Delta(x,E,r))= \M(\Delta_1).
  \]
Fix $\lambda= \lambda(n,c)=\max\{ t,2\},$ where $t>1$ satisfies
\[
\omega_{n-1} (\log 2 t)^{1-n}= c/2\,,
\]
and, for $\sigma>1, let$
\[\widetilde{\Delta}(x,E,r,\sigma)= 
\Delta(S^{n-1}(x,2r), E \cap (\overline{B}^n(x,r)\setminus
{B}^n(x,r/\sigma)); \overline{B}^n(x,2r)\setminus
{B}^n(x,r/\sigma)).\]
Then for every $\sigma\ge \lambda,$  
\[
\M(\widetilde{\Delta}(x,E,r,\sigma))\ge c/2.
\]

 \end{lemma}
 

 \begin{proof} The equality  $ \M(\Delta(x,E,r))= \M(\Delta_1)$ is a basic
 property of the modulus, see  
 \cite[p.33, Thm 11.3]{V71}. 
 By the subadditivity of the modulus in Lemma \ref{cgqm_5.3}, 
 \[
  \M(\Delta(x,E,r))= \M(\Delta_1)\le \M(\widetilde{\Delta}(x,E,r,\sigma))+  
  \M(\Delta_1 \setminus \widetilde{\Delta}(x,E,r,\sigma))
 \]
 and, by the choice of $\lambda$ and Lemmas \ref{cgqm_5.3} and \ref{cgqm_5.14},
 \[
  \M(\Delta_1 \setminus \widetilde{\Delta}(x,E,r,\sigma))
  \le c/2
 \]
for all $\sigma\ge \lambda.$ Finally,
 \[
  \M(\widetilde{\Delta}(x,E,r,\sigma))\ge \M(\Delta_1)-  
  \M(\Delta_1 \setminus \widetilde{\Delta}(x,E,r,\sigma)) \ge c- c/2=  c/2.
 \]
\end{proof}
  
  \medskip
  
  According to Lemma \ref{basicLem} we can find, under the above assumptions,
  a substantial portion of the set $E$ in the annulus 
  $ \overline{B}^n(x,r)\setminus B^n(x, r/\lambda)$ for $x \in E.$ 
  We need to use this type of  annuli for two disjoint sets $E$ and $F$, which are close 
  enough to each other and then to find a lower bound for the modulus 
  of the curve family joining the respective substantial portions of each set.
  These annuli are translated versions of each other and we can use Proposition
  \ref{ringProp} to find a common superannulus for both annuli and consider the
  joining curves in this superannulus. 
  To quantify this idea, we need a comparison principle of 
  the modulus of a curve family from \cite[p.61 Lemma. 5.35]{cgqm} , 
  \cite[p.182, Thm 5.5.1]{GMP}.

\begin{lemma}\label{cgqm5.35.}  (1) \cite[p.61 Lemma. 5.35]{cgqm}  
Let $G$ be a domain in $\overline{\R}^n$, let $F_j\subset G$,
$j=1,2,3,4$, and let $\Gamma_{ij}=\Delta(F_i,F_j;G)$, $1\le i,j\le 4$. Then
$$
\M(\Gamma_{12})\ge 3^{-n}\min\{\,\M(\Gamma_{13}),\;\M(\Gamma_{24}),\;
     \inf \M\bigl(\Delta(|\gamma_{13}|,|\gamma_{24}|;G)\bigr)\,\}\;,
$$
where the infimum is taken over all rectifiable curves $\gamma_{13}\in
\Gamma_{13}$ and $\gamma_{24}\in \Gamma_{24}$.

(2) \cite[p.63, 5.41  and 5.42]{cgqm} If $F_j \subset B^n(z,s) \subset G, j=1,2,3,4,$ and 
and there exists $t>0$ such that for all $\gamma_{13}\in
\Gamma_{13}$ and $\gamma_{24}\in \Gamma_{24}$
\[d( |\gamma_{13}|) \ge t s, \quad d(|\gamma_{24}|) \ge ts,\]
then there is a constant $v=v(n,b/a,t)$ such that
\[
\M(\Gamma_{12})\ge v \min\{\,\M(\Gamma_{13}),\;\M(\Gamma_{24})\;
  \}\;.
\]
\end{lemma}

\begin{proof} The first claim (1) is proved in the cited reference. 
The second claim also follows
easily from the cited reference but for clarity we include the details here.
 We apply the comparison principle of part (1)
to get a lower bound for $\M(\Gamma_{12})$. Because 
$F_j \subset B^n(z,s) \subset G, j=1,2,3,4,$
it follows from Lemma \ref{teichlow} that the
infimum in the lower bound of (1) is at least 
$v_1(n,t)\equiv\frac{1}{2}\tau_n(4m^2+4m),\, m=2/t,$ and
thus 
\[
\M(\Gamma_{12})\ge 3^{-n}\min \{\,\M(\Gamma_{13}),\;\M(\Gamma_{23}), v_1(n,t)\}.
\]
By Lemma \ref{cgqm_5.14}
$$
\max\bigl\{\,\M(\Gamma_{13}),\;\M(\Gamma_{23})\,\bigr\}\le
A \equiv\omega_{n-1}\Bigl(\log\frac{b}{a}\Bigr)^{1-n}\;
$$
and further
\begin{align*}
  \M(\Gamma_{12})&\ge 3^{-n}\min\Bigl\{\,\M(\Gamma_{13}),\;\M(\Gamma_{23}),\;
   {1\over A}\Bigl(v_1(n,t)\Bigr)\min\bigr\{
   \,\M(\Gamma_{13}),\;\M(\Gamma_{23})\,\bigr\}\Bigr\}\\
\noalign{\vskip2pt}
   &\ge v(n,b/a,t)\min\bigl\{\,\M(\Gamma_{13}),\;\M(\Gamma_{23})\,\bigr\}
\end{align*}
where $v(n,b/a,t)=3^{-n}\min\bigl\{\,1,\,\frac{1}{ A}\,v_1(n,t)\bigr\}$.

\end{proof}

   \medskip
   
The next lemma will be a key tool in Section \ref{sec7}.
It is based on three earlier results:
\begin{itemize}
\item[(a)]   the superannulus Proposition
\ref{ringProp},
\item[(b)]  the substantial subset selection Lemma \ref{basicLem},
\item[(c)]  the comparison principle
for the modulus of a curve family, Lemma \ref{cgqm5.35.}.   
\end{itemize}  

Recall the notation $R(x,  r',r)$ for $0<r<r'<+\infty$ 
and $x\in\R^n$ from \eqref{myRing}.

  \begin{lemma}\label{TwoSets}
  Let $n\ge 2, c >0$ and $\lambda=\lambda(n,c)\ge 2$ be as in Lemma \ref{basicLem}.
Let  $E, F \subset {\mathbb R}^n$  be   compact sets,  
$x,y\in {\mathbb R}^n,$ and  $r>0$ and 
suppose that
\[
 \M(\Delta(x,E,r) )= c >0, \quad  \M(\Delta(y,F,r) )= c >0.
\]
If $|x-y|<r/\lambda^2,$ then there exist constants 
$ \tau= \tau(n,c)\ge 2, d=d(n,c)>0$ such that with $w=(x+y)/2$
\[
\M(\Delta(E, F; R(w, \tau^2  r, r/ \tau^2)))\ge d .
\]
\end{lemma}
\begin{proof} By Proposition \ref{ringProp}, we see that
\[
R(x,r, r/\lambda) \cup R(y,r, r/\lambda) \subset R(w, \lambda^2 r,  r/\lambda^2)
\]
and that $ R(w, \lambda^2 r,  r/\lambda^2)$ is a common superannulus for the
two smaller annuli.
We apply the comparison principle for the modulus,  Lemma \ref{cgqm5.35.}(2), with
$$F_1= E \cap R(x,r, r/\lambda),  \quad F_2= F \cap R(y,r, r/\lambda), $$
$$F_3= S^{n-1}(x, 2r), \quad F_4= S^{n-1}(y, 2r).$$ 
Observe first that, in the notation of Lemma  
\ref{cgqm5.35.}, for all $\gamma_{13}\in
\Gamma_{13}$ and $\gamma_{24}\in \Gamma_{24}$
\[ d(|\gamma_{13}|)\ge r , \quad d(|\gamma_{24}|)\ge r 
\]
and that, by Proposition \ref{ringProp},
 $F_j \subset R(w, 2 \lambda^2 r, r/\lambda^2)$ for $j=1,2,3,4.$
By Lemma \ref{cgqm5.35.}(2), we have
$$
\M(\Delta(F_1, F_2; \R^n))\ge d_2(n,c) >0.
$$
Using Lemma \ref{basicLem} we obtain the numbers $d= d_2(n,c)/2$ and $\tau=\sqrt{2}\lambda.$
\end{proof}

\section{Hausdorff content and lower estimate of capacity} \label{HcontCap}


In this section, we discuss lower bounds for the capacity
in terms of the Hausdorff $h$-content. Our main
references are O. Martio \cite{M} and Yu.G. Reshetnyak \cite{R1}, \cite[pp.110-120]{R2}.
Reshetnyak also cites an earlier lemma of H. Cartan 1928 and gives its proof based
on the work of L.V. Ahlfors \cite{Ah1930} (cf. R. Nevanlinna \cite[p.141]{N}).
We give here a short review of the earlier relevant
results and, for the reader's benefit,
outline sketchy proofs.

We start with the next covering lemma \cite[p.~197]{Land}.
In the following, we denote by $\chi_E$ the characteristic function of a set 
$E\subset \R^n;$
that is, $\chi_E(x)=1$ if $x\in E$ and $\chi_E(x)=0$ otherwise.

\begin{lem}\label{lem:cl}
Let $n$ be an integer with $n\ge 2$ and let $A$ be a set in $\R^n.$
Suppose that a radius $r(x)>0$ is assigned for each point $x\in A$
in such a way that $\sup_{x\in A}r(x)<+\infty.$
Then one can find a countable subset $\{x_k\}$ of $A$
such that
\begin{equation}\label{eq:cl}
\chi_A(x)\le \sum_k \chi_{B_k}(x)\le N_n,\quad x\in\R^n,
\end{equation}
where $B_k=B^n(x_k,r(x_k))$ and $N_n$ is a constant depending only on $n.$
\end{lem}

The above inequalities mean that $A$ is covered by the family of balls
$\{B_k\}$ and
the number of overlapping of the covering is at most $N_n.$
It is an interesting problem to find the best possible number $N_n^*$ for
the constant $N_n$ in the above lemma.
For $\varphi\in(0,\pi],$
we will say that a subset $V$ of the unit sphere $S^{n-1}=\{x\in\R^n: |x|=1\}$
is $\varphi$-separated if the angle subtended by the two line segments $PO$ and $QO$
is at least $\varphi$ for distinct points $P, Q$ in $V.$
We denote by $\nu(n,\varphi)$ the maximal cardinal number of $\varphi$-separated
subsets $V$ of $S^{n-1}.$
For instance, $\nu(n,\pi)=2.$
It is clear that $\nu(n,\varphi)\le\nu(n,\varphi')$ for $0<\varphi'<\varphi.$
The proof of the above lemma in p.~199 of \cite{Land} tells us that
$$
N_n^*\le 1+\inf_{\varphi<\pi/3}\nu(n,\varphi).
$$
On the other hand, a standard compactness argument leads to the left continuity 
of $\nu(n,\varphi)$;
that is, $\nu(n,\varphi')\to \nu(n,\varphi)$ as $\varphi'\to\varphi^-.$
Hence, we have $N_n^*\le 1+\nu(n,\pi/3).$
Here we note that the number $\nu(n,\pi/3)$ is known as the
\textit{kissing number} $\kappa(n)$ in dimension $n$ \cite{b12}.
This number is closely related to other important issues such as sphere packing problems.
For instance, it is known that $\kappa(2)=6, \kappa(3)=12, \kappa(4)=24.$
However, it is difficult to determine $\kappa(n)$ in general.
The true value of $\kappa(5)$ is not determined up to the present.
By using the special nature of the lattices $E_8$ and $\Lambda_{24},$
Viazovska determined $\kappa(8)$ and, later with her collaborators, $\kappa(24)$
and won a Fields medal in 2022 \cite{Ok}.
In summary, we can state the following. 

\begin{lem}\label{lem:kn}
The minimal number $N_n^*$ of the bound $N_n$ in Lemma \ref{lem:cl}
satisfies the inequality $N_n^*\le \kappa(n)+1,$
where $\kappa(n)$ is the kissing number in dimension $n.$
\end{lem}

Let $h(r)$ be a measure function, that is, a monotone increasing continuous 
function on $0<r<+\infty$
with $h(r)\to0$ as $r\to0$ and $h(r)\to+\infty$ as $r\to+\infty.$
The $h$-Hausdorff content of a set $E\subset\R^n$ is defined as
$$
\Lambda_h(E)=\inf\left\{\sum_k h(r_k):
E\subset \bigcup_k B^n(a_k,r_k)\right\}.
$$
For a positive number $\beta>0,$  and for $h(r)=r^\beta,$ 
the $h$-Hausdorff content is called
the $\beta$-dimensional Hausdorff content and denoted by $\Lambda^\beta(E).$
Recall that the Hausdorff dimension $\dim E$ of $E$ is characterized as 
the infimum of $\beta>0$ with $\Lambda^\beta(E)=0.$

The next lemma constitutes a key step in the proof of Martio's Theorem
 \ref{thm:Mar79} below.
In the proof below we give explicit estimates of the relevant constants.

\begin{remark}\label{Cartan} The next lemma has a long history which goes back to
H. Cartan and L.V. Ahlfors  \cite{Ah1930}, \cite[p. 141]{N}.
Yu.G. Reshetnyak \cite{R1}, \cite[Lemma 3.7, p.115]{R2}, 
extended their two dimensional work to the case of $\mathbb{R}^n$ 
and applied the result to prove a lower bound for the capacity
in terms of the Hausdorff content.
O. Martio, in turn, made use of these results in his paper \cite{M} which
is one of our key references. 
\end{remark}

\begin{lem}[Lemma 2.6 in $\text{\cite{M}}$]\label{lem:Landkof}
Let $\sigma$ be a positive finite measure on $\R^n$ and $h$ be a measure function.
We denote by $T$ the set of those points $x\in\R^n$ for which the inequality
$\sigma(B^n(x,r))\le h(r)$ holds for all $r>0.$
Then $\Lambda_h(\R^n\setminus T)\le N_n^*\sigma(\R^n),$ where
$N_n^*$ is given in Lemma \ref{lem:kn}.
\end{lem}

\begin{pf}
We choose $r_0>0$ so that $h(r_0)=\sigma(\R^n)$ and let $A=\R^n\setminus T.$
For each $x\in A,$ by definition, there is a positive $r(x)$ such that
$\sigma(B^n(x,r(x)))>h(r(x)).$
Note that $r(x)\le r_0.$
We now apply Lemma \ref{lem:cl} to extract a countable set $\{x_k\}$ from $A$
so that $B_k=B^n(x_k,r(x_k))$ satisfy \eqref{eq:cl}.
Then
\begin{align*}
\Lambda_h(A)&\le \sum_k h(r(x_k))<\sum_k \sigma(B_k) \\
&=\sum_k\int \chi_{B_k}d\sigma
=\int \sum_k\chi_{B_k}d\sigma
\le N_n^* \sigma(\R^n).
\end{align*}
\end{pf}

By making use of the preceding lemma, Martio proved the following result.
For the proof, see Lemma 2.8 in \cite{M}.

\begin{lem}\label{lem:Martio}
Let $1<p\le n$ and $\alpha>0.$
For a function $u\ge 0$ in $L^p(\R^n)$ with support in $B^n(0,r_1),$
the set $F$ of points $x\in\R^n$ satisfying the inequality 
$$
v(x):=\int_{\R^n}u(y)|x-y|^{1-n}dm(y)
>\Omega_n^{1-1/p}\left(\frac{n-1}\alpha
\int_0^{r_1}h(t)^{1/p}t^{-n/p}dt+r_1^{1-n/p}\|u\|_p\right)
$$
admits the estimate $\Lambda_h(F)\le N_n(\alpha\|u\|_p)^p,$
where $N_n$ is the number in Lemma \ref{lem:cl}.
\end{lem}

Let
$$
K_n=2^{-n}\omega_{n-1}^n\Omega_n^{1-n}=\frac1{\Gamma(\frac n2+1)}
\left(\frac{n\sqrt\pi}2\right)^n.
$$
The following  theorem is a special case of Theorem 3.1 in \cite{M} when $p=n.$
Since an explicit form of  the constant $M_1$ is not given in \cite{M}, 
we give an outline of the proof with a concrete form of $M_1.$

\begin{thm}\label{thm:Mar79}
Suppose that a measure function $h(r)$ satisfies the inequality
$$
I(r)=\int_0^{2r}h(t)^{1/n}t^{-1}dt\le A h(r)^{1/n},\quad 0<r\le r_0,
$$
for some constants $A>0$ and $r_0>0.$
Let $E$ be a closed set in $\R^n.$
Then
\begin{equation} \label{capContent}
\frac{\Lambda_h(E\cap \Bbn(x,r))}{h(r)}\le M_1\cdot \capa(x, E, r),
\quad x\in\R^n,~ 0<r\le r_0,
\end{equation}
where $M_1$ is the positive constant given by
\begin{equation}\label{eq:M1}
M_1=  \max \left\{2^n A^n \left(\frac{n-1}{n}\right)^n\,N_n ,1/K_n \right\}.
\end{equation}
\end{thm}

\begin{pf}
We may assume $x=0$ and write $B_r=B^n(0,r)$ for short.
Because $E\cap B_r\subset B_{r'}, r<r',$ by the definition
of the Hausdorff content it is clear that $\Lambda_h(E\cap B_r)\le h(r).$
Since $M_1 \ge 1/K_n ,$
 the required inequality holds trivially when $\capa(0,E,r)\ge K_n.$
Thus we may assume that $\capa(0,E,r)<K_n.$
By the definition of capacity, for each $r>0$ and a small enough $\varepsilon>0,$ 
we may choose a smooth function
$w\ge0$ on $\R^n$ with support in $B_{2r}$ so that $w>1$ on $E \cap B_r$
and so that
$$
\|\nabla w\|_n^n=
\int_{\R^n}|\nabla w|^n dm
<\capa(B_{2r}, E\cap B_r)+\varepsilon
=\capa(0, E, r)+\varepsilon<K_n,
$$
where $dm$ denotes the Lebesgue measure.

We apply Lemma \ref{lem:Martio} with $r_1=2r$ and $p=n$ to the function 
$u=|\nabla w|/\omega_{n-1}$
and by the above inequality we may choose $\alpha$ so that
$$
\Omega_n^{1-1/n}\left(\frac{n-1}\alpha I(r)+\frac{\|\nabla w\|_n}{\omega_{n-1}}\right)=1.
$$
Since $\|\nabla w\|_n<K_n^{1/n},$ we have the inequality
$
\|\nabla w\|_n/\omega_{n-1}<\Omega_n^{1/n-1}/2,
$
which enables us to estimate $\alpha$ as
$$
\alpha=\frac{(n-1)I(r)}{\Omega_n^{1/n-1}-\|u\|_n}
<\frac{2(n-1)I(r)}{\Omega_n^{1/n-1}}
\le 2A(n-1)\Omega_n^{1-1/n}h(r)^{1/n}
$$
for $0<r\le r_0.$
By the representation formula (see \cite[(2.2)]{M})
$$
w(x)
=\frac1{\omega_{n-1}}\int_{\R^n}\frac{\nabla w(y)\cdot(x-y)}{|x-y|^n}dm(y),
$$
we have the inequality
$$
w(x)\le\frac1{\omega_{n-1}}\int_{\R^n}\frac{|\nabla w(y)||x-y|}{|x-y|^n}dm(y)
=\int_{\R^n}u(y)|x-y|^{1-n}dm(y)=v(x).
$$
Let $F$ be as in Lemma \ref{lem:Martio}.
In particular, we have $E\cap B_{r}\subset F$ and thus by Lemma \ref{lem:Martio}
$$
\Lambda_h(E\cap B_{r})\le \Lambda_h(F)\le  N_n(\alpha\|u\|_n)^n
\le N_n\, d_n \, h(r)(\capa(0,E,r)+\varepsilon)\,,
$$
where
$$
d_n= \frac{2^n A^n (n-1)^n}{n \Omega_n}\left(\frac{\Omega_n}{\omega_{n-1}} \right)^{n-1}
=2^n A^n (\frac{n-1}{n})^n.
$$
Letting $\varepsilon\to0,$ we obtain finally
$$
\Lambda_h(E\cap B_{r})\le 2^n A^n \left( \frac{n-1}{n} \right)^n\, N_n\,h(r)\,\capa(0,E,r).
$$
Now, in conjunction with Lemma \ref{lem:kn}, the required inequality follows.
\end{pf}

When $h(r)=r^\beta$ for some $0<\beta\le n,$ we obtain
$$
\int_0^{2r}h(t)^{1/n}t^{-1}dt=
\int_0^{2r}t^{\beta/n-1}dt=\frac n\beta (2r)^{\beta/n}
=A\, h(r)^{1/n},\quad
A=\frac n\beta 2^{\beta/n}.
$$
Hence, in this case, the constant $M_1$ in \eqref{eq:M1} may be expressed by
\begin{equation}\label{eq:M1b}
M_1(n,\beta)=\max\left\{2^{\beta+n}
\left(\frac{n-1}{\beta}\right)^n\,N_n^*,\, 1/K_n\right\}.
\end{equation}
For instance, when $n=2,$ we have $N_2^*=7$ and thus
$M_1(2,\beta)=28( 2^\beta/\beta^2).$


\section{Uniform perfectness and capacity} \label{sec:UP}
In this section, we study the connection between potential theoretic thickness
of sets, as expressed in terms of capacity, and uniform perfectness.
The notion of uniform perfectness was first used by Beardon and Pommerenke
 \cite{BP78} in two dimensions.
Later, Pommerenke \cite{Pom79} found a characterization of uniform perfectness
in terms of the logarithmic capacity.
One of the main results of J\"arvi and Vuorinen \cite{JV96}
was a characterization of uniform perfectness for 
the general dimension in terms of the quantity $\capa (x,E,r)$
defined in Section \ref{HcontCap}. The novel feature in our work is to give
an explicit form for  $\capa (x,E,r)$ in terms of the dimension $n$ and the 
parameter $c$ of $\UP_n(c)$ sets.
Then we will prove Lemma \ref{introHDimBd} and Theorem \ref{introCapBd} given in 
Introduction.

\begin{definition}\label{UPcdef}
For $c \in (0,1),$ let $\UP_n(c)$ denote the collection of compact sets 
$E$ in $\mathbb{R}^n$ with $\card(E)\ge 2$ satisfying the condition
\begin{equation*}
\{x\in E: c r<|x-a|<r\}\ne\emptyset \quad\text{for all}~ a\in E
\aand 0<r< d(E)/2\,.
\end{equation*}
A set is called {\it uniformly perfect} if it is of class $\UP_n(c)$ for
some $c.$
\end{definition}

\begin{lem}[Proposition 7.4 in \cite{Sugawa98}]\label{sug98}
Let $I$ be the index set $\{1,2,\dots, p\}$ and $0<c<1$ and $r>0.$
Suppose that a sequence of families of closed balls
$$\bB_{i_1i_2\dots i_k}=\Bbn(a_{i_1i_2\dots i_k}, c^k r), \quad i_1, i_2, \dots, i_k\in I,$$
for $k=1,2,3,\dots,$ satisfies the following two conditions:
\begin{enumerate}
\item[(i)]
$\bB_{i_1i_2\dots i_{k-1}i}\cap \bB_{i_1i_2\dots i_{k-1}j}=\emptyset$
for $i,j\in I$ with $i\ne j$ and $i_1,\dots, i_{k-1}\in I.$
\item[(ii)]
$\bB_{i_1i_2\dots i_k}\subset \bB_{i_1i_2\dots i_{k-1}}$
for $i_1,\dots, i_{k}\in I.$
\end{enumerate}
Then the set
$$
K=\bigcap_{k=1}^\infty\bigcup_{(i_1,\dots,i_k)\in I^k} \bB_{i_1i_2\dots i_k}
$$
satisfies the inequality $\Lambda^\beta(K)\ge r^\beta/(p\cdot 3^n),$
where $\beta=-(\log p)/\log c,$ and its Hausdorff dimension is $\beta.$
\end{lem}

This lemma has the following important corollary.

\begin{cor}
Let $E\in\UP_n(c)$ for some $0<c<1.$ Then the Hausdorff dimension $\dim_H(E)$ of
$E$ is at least $(\log 2)/\log(3/c),$ which is independent of the dimension $n.$
\end{cor}

\begin{rem}\label{CantorDim} 
Recall that the Cantor middle-third set $K$ has Hausdorff dimension
$(\log 2)/\log 3\approx 0.63093$ \cite[p.60]{Mat}.
One can check that $K\in \UP_2(2/5)$ and the number $2/5$ cannot be
increased.
The above corollary thus implies that $\dim_H(K)\ge (\log
2)/\log(15/2)\approx 0.34401.$
\end{rem}

\begin{thm}\label{HDimBd}
Let $E\in\UP_n(c)$ for some $0<c<1.$
Then for every $a\in E,~ a\ne\infty,$
\begin{equation} \label{upLam}
\Lambda^\beta(E\cap\Bbn(a,r))\ge \frac{r^\beta}{2\cdot 3^n}, \quad 0<r<d(E)/2,
~\text{where}~ \beta=\frac{\log 2}{\log(3/c)}.
\end{equation}
\begin{equation} \label{upCap}
\capa(x,E,r))\ge \frac{1}{2\cdot 3^n M_1(n, \beta)}, \quad 0<r<d(E)/2,
~\text{where}~ \beta=\frac{\log 2}{\log(3/c)}
\end{equation}
and where $M_1$ is as in \eqref{eq:M1b}.
\end{thm}

\begin{pf}
Set $\bB=\Bbn(a,r).$
Let $a_1=a$ and take a point, say $a_2,$ from the set 
$\{x\in E: 2cr/3<|x-a|<2r/3\},$
which is non-empty by assumption.
Then we define $\bB_i=\Bbn(a_i,r_1)$ for $i=1,2$ and $r_1=cr/3.$
Since $|a_1-a_2|>2cr/3=2r_1,$ we have $\bB_1\cap\bB_2=\emptyset.$
Also, by $|a-a_2|+r_1<2r/3+cr/3<r,$ we confirm that $\bB_2\subset \bB.$

Next we let $a_{i1}=a_i$ for $i=1,2$ and choose $a_{i2}$
from the set $\{x\in E: 2cr_1/3<|x-a_i|<2r_1/3\}.$
Then we define $\bB_{i_1i_2}=\Bbn(a_{i_1i_2}, r_2)$
for $i_1,i_2\in I=\{1,2\}$ and $r_2=cr_1/3=(c/3)^2 r.$
We can proceed inductively to define families of disjoint closed
balls $\bB_{i_1i_2\dots i_k}=
\Bbn(a_{i_1i_2\dots i_k}, r_k)~(i_1, i_2, \dots, i_k\in I)$
for $k=1,2,3,\dots,$ and $r_k=(c/3)^k r,$ in such a way that
$\bB_{i_1i_2\dots i_k}\subset \bB_{i_1i_2\dots i_{k-1}}.$
We finally set
$$
K=\bigcap_{k=1}^\infty\bigcup_{(i_1,\dots,i_k)\in I^k} \bB_{i_1i_2\dots i_k}.
$$
Then $K$ is a Cantor set and satisfies the inequality
$\Lambda^\beta(K)\ge r^\beta/(2\cdot 3^n)$ by Lemma \ref{sug98}.
Since $K\subset E\cap\Bbn(a,r)$ by construction, the proof of 
\eqref{upLam} is complete.

The proof of \eqref{upCap} follows from Theorem \ref{capContent} and 
\eqref{upLam}.
\end{pf}

\medskip

\begin{nonsec}{\bf Proofs of  Lemma \ref{introHDimBd} and Theorem \ref{introCapBd}.}
Theorem \ref{HDimBd} also proves Lemma \ref{introHDimBd} and Theorem \ref{introCapBd}.
\hfill $\square$
\end{nonsec}

Next we estimate the parameter $c$ in terms of a $\capa(x,E,r)$ lower bound.

\begin{lem}
Suppose that a closed set $E$ in $\Sn$ satisfies $\capa(x,E,r)\ge \sigma$
for $x\in E$ and $0<r\le d(E)$ for a constant $\sigma>0.$
Then $E\in\UP_n(c)$ for any $c>0$ with
$
c\le 2\exp \{-\left(\omega_{n-1}/\sigma\right)^{1/(n-1)} \}.
$
\end{lem}

\begin{pf}
Suppose that $E\cap R(a,\alpha r, r)=\emptyset$ for some 
$0<r<d(E)/2$ and $a\in E.$
Then $E\cap\Bbn(a,r)\subset\Bbn(a,\alpha r).$
By Lemma \ref{cgqm_5.14}, we now obtain
$$
\sigma\le
\capa(B^n(a,2r),E\cap\Bbn(a,r))
=\capa(B^n(a,2r),\Bbn(a,\alpha r))=
\omega_{n-1}\left(\log\frac{2}{\alpha}\right)^{1-n},
$$
which implies
$
\alpha\ge 2\exp \{-\left(\omega_{n-1}/\sigma\right)^{1/(n-1)} \}.
$
\end{pf}

\begin{nonsec}{\bf Equivalent characterization of uniform perfectness.} According to
\cite{JV96} a closed set $E\subset \mathbb{R}^n$ with ${\rm card}E \ge 2$ is $\alpha$-
uniformly perfect,  $\alpha>0,$ if there is no ring domain $D$ separating $E$ with 
${\rm mod} D> \alpha.$  The following lemma  shows that
this notion is quantitatively equivalent to our notion of $\UP_n(c),$ with explicit 
constants.
\end{nonsec}

\begin{lem} Let $E$ be a closed set in $\R^n$ containing at least two points.

(1) If $E$ is $\alpha$-uniformly perfect, then $E\in\UP_n(e^{-\alpha})$.

(2) If $E\in\UP_n(c),$ then $E$ is $\alpha$-uniformly perfect for
$\alpha<A_n+\log(3/c),$ where $A_n$ is a positive constant depending only on $n$. 
\end{lem}

\begin{proof}
This result is contained in the proof of Theorem 3.3 in \cite{GSV}.
\end{proof}

According to \cite{GSV}, the constant $A_n$ admits the following majorant for $n\ge 2$
\[
A_n \le 2 \log \frac{(1+\sqrt{2})\lambda_n}{2}\,;\quad  4 \le \lambda_n \le 2^{n/(n-1)}
e^{n(n-2)/(n-1)}\,.
\]
The constant $\lambda_n$ is the so called Gr\"otzsch constant \cite{GMP, GSV, HKV20}.
Observe that we take $r$ in $0<r<\diam( E) /2$ in the above definition \ref{UPcdef} 
of $\UP_n(c)$, whereas in \cite{GSV} it is required that $0<r<\diam( E).$

%
%
%
%
%

\medskip
%

\section{ Proof of Theorems \ref{thm_1.3} and \ref{newUP1}}\label{sec7}

In this section our goal is to prove  one of the main
results of this paper, Theorem \ref{thm_1.3}, which gives a lower bound for
\[
{\rm cap}(G,E)
\] 
when $G \subset {\mathbb{R}^n}$ is a domain and $E \subset G$ is a compact set
and both $E$ and $\partial G$ are uniformly perfect. 
The proof is based on the results given in earlier sections and
it is divided
into three cases: (a) $d(E)/d(E,\partial G)$ is small (Lemma \ref{low_bd}),  (b) 
$d(E)/d(E,\partial G)$ is large (Lemma \ref{lem_moduliforE}),  
(c) neither (a) nor (b) holds. 
These three cases form the logical structure of the proof of Theorem 
\ref{ideaToshi} which immediately yields the proof of  Theorem \ref{thm_1.3}.
For the case (b) we apply Proposition \ref{prop_rhoE4bound} and Lemma \ref{TwoSets} to construct 
a sequence of separate annuli  with the parameter $\lambda$ adjusted 
so that each annulus contains a substantial portion of both $E$ and $\partial G.$

Then using  Theorem \ref{thm_1.3} we also prove Theorem  \ref{newUP1}.

In  Lemma \ref{basicLem} we proved that, for a compact set $E \subset \R^n$ of
positive capacity, the condition
${\rm cap}(x,E,r)=c>0$ implies the existence of $\lambda =\lambda(n,c)$ such that the set
$$ E \cap (\overline{B}^n(x,r) \setminus {B}^n(x,r/\lambda)) $$
is quite substantial. Now for a uniformly perfect set
  $E$  and    $x \in E,$  we see by Theorem \ref{HDimBd} that for $0<r<d(E)$ the sets
 $$ E \cap (\overline{B}^n(x,r_k) \setminus {B}^n(x,r_{k+1})) $$
 are substantial for all $k=1,2,\dots,$ where $r_1=r, r_{k+1}= r_k/\lambda.$ Observe that
 these sets are subsets of separate annuli centered at $x.$

\bigskip

Our first result in this section, Lemma \ref{low_bd}, yields a lower bound for the modulus of the family of all
curves joining for a pair of compact
sets $F_1$ and $F_2$, in terms of the respective capacities, the dimension
$n$ and a set separation  parameter $t \in (0,1/2).$
This result is a counterpart of
Lemma  \ref{ringcap} which gives a similar lower bound for a pair of
continua $E$ and $F.$ The parameter $t$ now plays the role 
of $d(E,F)/\min\{d(E),d(F)\}.$

\begin{proposition} \label{sugProp}
 Let $K>1, n\ge 2,$ and
\[
g(x)= \log(1+x) \left(\log\frac{K}{x}\right)^{n-1}, \quad 0<x\le 1\,.
\]
Then $g(x)\le K ((n-1)/e)^{n-1}$ for all $ 0<x\le 1 \,.$ 
In particular, for $K> 1>x>0$
\begin{equation} \label{sugFormula}
\left( \log \frac{K}{x}\right)^{1-n} \ge  
\frac{1}{K} \left(\frac{e}{n-1}\right)^{n-1} \log(1+x)\,.
\end{equation}
\end{proposition}

\begin{proof}   Let $h:[0,\infty)\to [0,\infty)$ be defined by 
$h(u)=e^{-u} u^{n-1}.$ Then $h(0)=0, h(\infty)=0,$ and $h$ has 
its only maximal value at $u=n-1,$ equal to $h(n-1) = ((n-1)/e)^{n-1}.$
Setting $u= \log(K/x)$ and applying $\log(1+t)\le t, t\ge 0,$ yields
\[
g(K e^{-u})= u^{n-1} \log(1+ K e^{-u})\le K h(u) \le K h(n-1)\,.
\]
For  $K> 1$ let
\begin{equation*} 
s(n,K) = \inf_{0<x<1} \frac{\left( \log (K/x)\right)^{1-n}}{\log(1+x)}.
\end{equation*}
The above upper bound for the function $h$ shows that for all $x \in(0,1]$
\begin{equation*} 
\left( \log (K/x)\right)^{1-n} \ge s(n,K) \log(1+x)\ge 
\frac{1}{K} (e/(n-1))^{n-1} \log(1+x)
\end{equation*}
which completes the proof.
\end{proof}

\bigskip

\begin{lemma}\label{low_bd} Let $t \in (0,1/2]$ and let  
$F_j \subset\overline{B}^n((j-1)e_1, t)$ be compact sets with
$$\delta_j={\rm cap}((j-1)e_1,F_j, t)>0,\quad j=1,2.$$
Then there exists a constant $\mu_n>0$ depending only on $n,$ such that
$$\M(\Delta(F_1,F_2; \mathbb{R}^n)) \ge \mu_n\, \min\{\delta_1,\delta_2\} 
\log(1+t) \,.$$
\end{lemma}

\begin{proof} Let $G=B^n(\frac{1}{2} e_1, \frac{3}{2}),$ let $F_3= \partial G,$
and let $\Delta_{j 3}= \Delta(F_j, F_3;G), j=1,2\,.$ 
Then $G \subset B^n(e_1,2),$ $G \subset B^n(0,2)$
and Lemma \ref{marsarbd} with the triple of radii $\{t,2t,2\}$ yields
$$ \M(\Delta_{j 3})= {\rm cap}(G,F_j)\ge u(t) \delta_j\,, \quad u(t) \equiv \left(\frac{\log 2}{\log(2/t)}\right)^{n-1}\,, j=1,2.$$
By the choice of $G$ for all $\gamma_j \in \Delta_{j 3}$ we have
$$d(|\gamma_j|)\ge 1-t \ge 1/2\,,\quad j=1,2,$$
and by Lemma \ref{teichlow}
$$
\M(\Delta(|\gamma_1|,|\gamma_2|; G)) \ge \frac{1}{2} \tau_n(4m^2+4m)
$$
where $m=2/(1-t)\le 4$ and $4m^2+4m\le 80.$ Because 
$\M(\Delta_{j 3})\le A\equiv\omega_{n-1}(\log 2)^{1-n},$ we see that
$\tau_n(80)/2 \ge (\tau_n(80)/(2A)) \min\{\delta_1, \delta_2\} $
and hence by Lemma \ref{cgqm5.35.}, \cite[5.35]{cgqm}
$$
\M(\Delta(F_1,F_2; \mathbb{R}^n)) \ge 3^{-n}
 \min\{ \M(\Delta_{1 3}), \M(\Delta_{2 3}),\tau_n(80)/2\}
$$
$$
\ge 3^{-n} \min \left\{1, \frac{\tau_n(80)}{2A}\right\} u(t) \min\{\delta_1, \delta_2\}.
$$
Finally, we estimate $u(t)$ using \eqref{sugFormula}
$$
u(t)\ge (\log 2)^{n-1} \frac{1}{2} \left( \frac{e}{n-1}\right)^{n-1} \, \log(1+t)\,.
$$
In conclusion, we can choose
$$ \mu_n=3^{-n} \min \left\{1,\frac{\tau_n(80)}{2A}\right\} (\log 2)^{n-1} \frac{1}{2} \left( \frac{e}{n-1}\right)^{n-1} \,. $$
\end{proof}

\begin{lemma}\label{lem_moduliforE}
Let $G\subset \R^n$ be a  domain and let $E\subset G$ be a compact set, and  suppose that for all 
$r\in (0, \min\{d(\partial G),d(E)\})$ and for all $z_1\in E, z_2\in \partial G$
\[   \capa(z_1, E,r) \ge \delta\,, \quad  \capa(z_2, \partial G,r) \ge \delta>0\, . \]
Fix $x,y \in E,z_0\in \partial G$ such that
 $|x-z_0|=d(E,\partial G), |x-y|\ge d(E)/2$,
let
 $\lambda=\lambda(\delta,n)\ge4$ be the number given by Lemma \ref{TwoSets}, and denoted
 there as $\tau^2.$ 
If $d(E)/d(E,\partial G)>\lambda^{2p}+2$ for some integer $p\geq1,$ then
\begin{align*}
\M(\Delta(E,\partial G;G))\geq d(n,2) p \, \delta/4    
\end{align*}
where the constant $d$ is as in Lemma \ref{TwoSets}.
\end{lemma}
\begin{proof} Let $w=(x+z_0)/2$ and
consider the separate curve families in the annuli of Proposition \ref{prop_rhoE4bound}(2) 
and apply Lemma \ref{TwoSets}. 
\end{proof}

\begin{theorem} \label{ideaToshi}
Let $G,$ $E $ and $\delta>0$ be as in Lemma \ref{lem_moduliforE}. Then there exists
a constant $s>0$ depending only on $n$ and $\delta$ such that
\[
\capa(G,E) \ge  s \, \log(1+d(E)/d(E, \partial G)) \,.
\]
\end{theorem}

\begin{proof} 
Let $\lambda= \lambda(n,\delta)\ge 4 $ be as in Lemma \ref{TwoSets} denoted by $\tau^2$.

{\it  Case A:}
In the case $d(E)/d(E, \partial G)\leq1\slash2$, the proof follows from 
Lemma \ref{low_bd} with  a constant $c_A=\mu_n\,\delta .$ Indeed, fix
$a \in E, b\in \partial G$ with $|a-b|=d(E,\partial G)$ and with
$t=d(E)/d(E, \partial G) $ and $F_1=E \cap B^n(a, d(E)), 
F_2=\partial G \cap B^n(b, d(E)) $ apply Lemma  \ref{low_bd}. Observe that $d(\partial G )
> d(E)$ and the class of $\UP_n(c)$ is invariant under similarity transformations.

{\it Case B:} Consider next the case $t\equiv d(E)/d(E, \partial G)\ge t_0$ where 
$t_0=t_0(n,\delta)>1$
is defined as the number with $\log(1+t_0)= \log(\lambda^6+2)\,.$ Then 
$\log(1+t)\ge  \log(\lambda^{2p}+2)\,$ for some integer $p\ge3.$
Fix an integer $p\geq3$ with
\begin{align*}
\log(\lambda^{2(p+1)} +2)>\log(1+ t)\geq\log(\lambda^{2p}+2).    
\end{align*}
Then by Proposition \ref{quadProp}
\begin{align*}
2(p+1) \log\lambda  \geq\log(t-1)\ge \frac{1}{2} \log(1+t)+ 2 \log \lambda,    
\end{align*}
for all $t \ge \lambda^4+ 2\lambda^2 +1.$ Because $\lambda >3$ we see that
$
t_0=\lambda^6+1 \ge 3 \lambda^5 \ge \lambda^4+ 2\lambda^2+1
$
and hence for all $t\ge t_0$
\[
p \ge \frac{1}{4 \log \lambda}\log(1+t)
\]
and it follows
from Lemma \ref{lem_moduliforE} that
\begin{align*}
\capa(G,E)\geq d(n,2) p \delta/4\geq (d(n,2)\, \delta/4)\, \frac{1}{4 \log \lambda} \log(1+t) =c_B \log(1+t).
\end{align*}

{\it Case C:} Finally, we consider the case 
$d(E)/d(E, \partial G)\in (1/2, t_0)$ where $t_0=t_0(n,\delta)$ is as in Case B. Then by    Case B, $t_0\ge 3$ and
\begin{equation}\label{dEbd}
d(E, \partial G)/2 \le d(E)\le t_0 d(E, \partial G)\equiv T.
\end{equation}
Fix $x_0 \in E$ and $z\in \partial G$ with
$|x_0-z|=d(E,\partial G)$ and observe that $E \cup B^n(z, 2 d(E, \partial G)) \subset B^n(x_0,T).$ For the application of the comparison principle of Lemma \ref{cgqm5.35.}  employ the following notation
\[
F_1= E,\quad F_2= \overline{B}^n(z, d(E,\partial G)) \cap \partial G, \quad F_3=  S^{n-1}(x_0, 2 T).
\]
Let
\[
\Gamma_{13}= \Delta(F_1, F_3;{\mathbb R}^n), \quad {\Gamma}_{23}= \Delta( F_2, F_3; {\mathbb R}^n  ).
\]
Observe first that $F_1 \cup F_2 \subset \overline{B}^n(x_0,T)$ and $F_3 \cap B^n(x_0,2T) = \emptyset.$
Then  by  \cite[Remark 7.8, Lemma 7.1 (2)]{HKV20} and Lemma \ref{cgqm5.35.}(2)
\[
 \M(\Delta(E, \partial G; G)) =  \M(\Delta(E, \partial G;{\mathbb R}^n  ))\ge \M(\Delta(F_1, F_2;{\mathbb R}^n  ))
 \ge d(n,2) \min \{  \M( \Gamma_{13}), \M( \Gamma_{23})\}.
\]
By Lemma \ref{marsarbd} and \eqref{sugFormula}
\[
\M( \Gamma_{13})\ge \delta \left( \frac{\log 2}{\log(2T/d(E))}\right)^{n-1}\ge d_3
\log(1+ d(E)/(2T))
\]
\[
\ge \frac{d_3}{2 t_0}\log(1+ d(E)/d(E,\partial G))
\]
where \eqref{dEbd} and Bernoulli's inequality were applied. 
In the same way we have
\[
\M( \Gamma_{23})\ge \delta \left( \frac{\log 2}{\log(2T/d(E))}\right)^{n-1}\ge d_3
\log(1+ d(E)/(2T)).
\]
Because ${ T}/{ d(E, \partial G)}=t_0$ by Lemma \ref{marsarbd} and \eqref{sugFormula}
we obtain
$$ \M( \Gamma_{23})\ge c ((\log 2)/\log ((2T+d(E,\partial G))/d(E,\partial G)))^{n-1} $$
\[ \ge d_4 \log(1+ \frac{d(E)/t_0}{(2t_0+1)d(E,\partial G)}) \ge c_C
\log(1+ d(E)/d(E,\partial G)).\]

With the constant $\min\{ c_A, c_B, c_C\}$ we have proved the theorem.
\end{proof}

\begin{nonsec}{\bf Proof of Theorem \ref{thm_1.3}.}
{\rm The proof follows from Theorem \ref{ideaToshi}.}
\hfill$\square$
\end{nonsec}

\begin{nonsec}{\bf Proof of Theorem \ref{newUP1}.}
{ Suppose first that $\partial G$ is uniformly perfect and let $z \in G.$ Then
\[
d(E)/d(E, \partial G) \ge 2 \alpha /(1-\alpha)
\]
for $E=  \overline{B}^n(z, \alpha d(z, \partial G))$ and by Theorem  \ref{thm_1.3}
\[
u_{\alpha}(z)= {\rm cap}(G,E) \ge s \log(1+  2 \alpha /(1-\alpha)).
\]

For the proof of the converse implication we may assume by Theorem  \ref{uHarnack}
without loss of generality that $\alpha= 1/2$ and
 suppose that $u_{1/2}(z)\ge \gamma >0,$ for all $z \in G$ and write 
$F= \overline{\mathbb R}^n \setminus G$ for short. Take $c\in(0,1)$ so that
\[
-\log c=\left(\frac{2^n\omega_{n-1}}{\gamma}\right)^{1/(n-1)}.
\]
We now show that the annulus $cr<|x-a|<r$ meets $F$
for all $a\in F$ and $0<r<+\infty.$
On the contrary, we suppose that
there exist $a\in F\setminus\{\infty\}$ and $r>0$
such that the annulus $A=R(a, r,cr)$ separates $F.$
It is easy to see that $a$ is not an isolated point of $F.$
Thus, by decreasing $r$ if necessary, we may assume that 
there is a point $\xi\in\partial G$ with $|\xi-a|=cr.$
Let $x_0=a+s(\xi-a),$ where $s=1/\sqrt{c}.$
Then $|x_0-a|=\sqrt{c}r$ and thus $x_0\in A.$
Since $\sqrt{c}-c\le 1-\sqrt c,$ we have
$$\delta:=d(x_0,\partial G)=|x_0-\xi|=d(x_0,\partial A)=|x_0-a|-cr=(s-1)cr,$$ 
Set $C_0=\{x: |x-a|\le cr\},~ C_1=\{x: |x-a|\ge r\}$ and $B=\Bbn(x_0,\delta/2).$
Then 
$$\Delta(B,\partial A;A)\subset\Delta(B,C_0;A)\cup \Delta(B,C_1;A)$$ 
and hence
$$
\capa(A,B)\le \mM(\Delta(B,C_0;A))+\mM(\Delta(B,C_1;A))
\le \mM(\Delta(B,C_0;\Sn))+\mM(\Delta(B,C_1;\Sn)).
$$
Let $R_j=\Sn\setminus(B\cup C_j)$ for $j=0,1.$
Since $A_0=R(a, \sqrt{c}r, cr)$ is a subring of $R_0,$ we have
$$\mod(R_0)\ge \mod(A_0)=\log s.$$
Similarly, since $A_1=R(x_0,r-\sqrt{c}r-\delta/2, \delta/2)$ 
is a subring of $R_1,$
we have $\mod(R_1)\ge \mod(A_1).$
We now observe that
$$
\mod(A_1)=\log\frac{r-\sqrt{c}r-\delta/2}{\delta/2}
=\log\frac{s^2cr-scr-(s-1)cr/2}{(s-1)cr/2}=\log(2s-1)>\log s.
$$
In view of the relation 
$\mM(\Delta(B,C_j;\Sn))=\omega_{n-1}/(\mod(R_j))^{n-1}$ for $j=0,1,$
we see that
$$
\capa(A,B)\le \frac{\omega_{n-1}}{(\mod(R_0))^{n-1}}
+\frac{\omega_{n-1}}{(\mod(R_1))^{n-1}}
<\frac{2\omega_{n-1}}{(\log s)^{n-1}}=\frac{2^n\omega_{n-1}}{(-\log c)^{n-1}}
$$
Since $\Delta(B,\partial G;G)>\Delta(B, \partial A;A),$
 Lemma \ref{cgqm_5.3} yields
$$
\capa(G,B) \le \capa(A,B)<\frac{2\omega_{n-1}}{(-2\log c)^{n-1}}.
$$
By assumption, $\capa(G,B)=u_{1/2}(x_0)\ge \gamma$ so that
$(-\log c)^{n-1}<2^n\omega_{n-1}/\gamma,$ which is impossible by the
choice of $c.$
Hence, we have shown the required assertion and conclude that $F\in\UP_n(c)$.
\hfill$\square$}
\end{nonsec}

\section{Whitney Cubes and Uniform Perfectness}

Next, we will study the condenser capacity by using Whitney cubes.

\begin{nonsec}{\bf Whitney decomposition. } \label{Wdecmp}
If $G\subset\R^2$ is a bounded domain, we can clearly represent it as 
a countable union of non-overlapping closed squares. 
By the Whitney decomposition theorem, we can choose these squares 
${Q^k_j }\,, k \in  \mathbb{Z}, 0\le j \le  N_k\,, $ so that they 
have pairwise disjoint interiors and sides parallel to the coordinate axes 
and the following properties are fulfilled:
\begin{enumerate}
\item[(1)] every cube $Q^k_j\,,$ $0\le j \le  N_k\,, $ has sidelength $2^{-k}\,,$
\item[(2)] $G = \cup_{k,j} Q^k_j\,,$
\item[(3)] $d(Q^k_j ) \le d(Q^k_j, \partial G) < 4 d(Q^k_j )\,.$
\end{enumerate}
These kinds of squares ${Q^k_j}$ are \emph{Whitney squares} and this definition 
can be clearly extended to the general case $G\subset\R^n$, $n\geq2$, so that 
we will have $n$-dimensional closed hypercubes called \emph{Whitney cubes} 
instead of just squares \cite[Thm 1, p. 167]{ST}. Note that the Whitney cubes of
a domain $G$ resemble in a way hyperbolic balls of $\B^n$ with a constant
radius (Fig. 2).
\end{nonsec}

\begin{figure}[H] %
\centerline{
\scalebox{0.8}{\includegraphics[trim=0 0 0 0,clip]{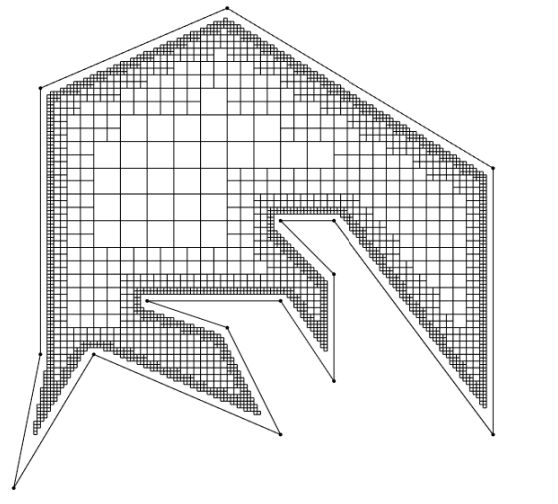}}
}
\caption{Whitney decomposition of a polygon. The picture was generated
with software by
 D. E. Marshall.}
\label{fig:Wsquares}
\end{figure}

\begin{nonsec}{\bf Whitney cubes and $u_\alpha\,.$}
For a domain $G\subset\R^n$, $\alpha\in(0,1)$ and $z\in G$, let 
\begin{align*}
u_\alpha(z)=\capa(G,\overline{B}^n(z,\alpha\,d(z,\partial G)))   
\end{align*}
and $G = \cup Q^k_j$ be a Whitney decomposition. Then $Q^k_j$ 
has a side length $2^{-k}$ and
\begin{align*}
d(Q^k_j) \le d(Q^k_j, \partial G) < 4 d(Q^k_j )=4\cdot2^{-k}\sqrt{n}.    
\end{align*}
If $m^k_j$ is the midpoint of $Q^k_j$, then clearly 
\begin{align}
\label{ineq_1_wc}
&B^n\left(m^k_j,\frac{d(Q^k_j)}{2\sqrt{n}}\right)\subset 
Q^k_j\subset\overline{B}^n\left(m^k_j,\frac{d(Q^k_j)}{2}\right),\\
\label{ineq_2_wc}
&\left(1+\frac{1}{2\sqrt{n}}\right)d(Q^k_j)\leq
d(m^k_j,\partial G)\leq \left(4+\frac{1}{2}\right)d(Q^k_j).
\end{align}
From \eqref{ineq_1_wc} and \eqref{ineq_2_wc}, we see that 
\begin{align*}
&u_\gamma(m^k_j)\leq\capa(\R^n\setminus Q^k_j,\partial G),
\quad
\gamma=\frac{d(Q^k_j)\slash(2\sqrt{n})}{9d(Q^k_j)\slash2}=\frac{1}{9\sqrt{n}},\\
&u_\eta(m^k_j)\geq\capa(\R^n\setminus Q^k_j,\partial G),
\quad
\eta=\frac{d(Q^k_j)\slash2}{(1+1\slash(2\sqrt{n}))d(Q^k_j)}
=\frac{\sqrt{n}}{1+2\sqrt{n}}.
\end{align*}
By Lemma \ref{lemma_moduli_uv}, we can now observe that 
\begin{align*}
u_\gamma(m^k_j)\leq\capa(\R^n\setminus Q^k_j,\partial G)
\leq d_1u_\gamma(m^k_j),\quad
d_1=d_1(n)=\left(\frac{9\sqrt{n}}{1+2\sqrt{n}}\right)^{n-1}.
\end{align*}
\end{nonsec}

The next lemma follows from the the above observations.

\begin{lemma}\label{lemma_A_wc}
For a cube $Q^k_j$ in the Whitney decomposition of the domain $G\subset\R^n$, 
let $m^j_k$ be its midpoint. Then there exists a constant $d_1>0$ 
only depending on $n$ such that
\begin{align*}
u_\gamma(m^k_j)\leq\capa(\R^n\setminus Q^k_j,\partial G)
\leq d_1u_\gamma(m^k_j)   
\end{align*}
holds with $\gamma=1\slash(9\sqrt{n})$.
\end{lemma}

\begin{lemma}\label{lemma_B_wc}
Let $Q^k_j$ be as in Lemma \ref{lemma_A_wc}. Then there exists a constant $d_2>0$ 
only depending on $n$ such that, for all $z\in Q^k_j$,
\begin{align*}
{u_\gamma(z)}/{d_2}\leq\capa(\R^n\setminus Q^k_j,\partial G)\leq d_2u_\gamma(z)   
\end{align*}
holds with $\gamma=1\slash(9\sqrt{n})$.
\end{lemma}
\begin{proof}
By Theorem \ref{uHarnack}, $u_\gamma(z)$ satisfies the Harnack inequality in $G$. 
Therefore by \cite[Lemma 6.23]{HKV20}, there is a constant $d_3>0$ only 
depending on $n$ such that
\begin{align*}
{u_\gamma(z)}/{d_3}\leq u_\gamma(m^k_j)\leq d_3u_\gamma(z).   
\end{align*}
The proof follows now from Lemma \ref{lemma_A_wc}.
\end{proof}

\begin{nonsec}{\bf Proof of Theorem \ref{thmWUP}.}
\label{proof_forNewUP1}
The proof follows from Lemma \ref{lemma_B_wc} and Theorem \ref{newUP1}.
\hfill$\square$
\end{nonsec}


\def\cprime{$'$} \def\cprime{$'$} \def\cprime{$'$}
\providecommand{\bysame}{\leavevmode\hbox to3em{\hrulefill}\thinspace}
\providecommand{\MR}{\relax\ifhmode\unskip\space\fi MR }
\providecommand{\MRhref}[2]{%
  \href{http://www.ams.org/mathscinet-getitem?mr=#1}{#2}
}
\providecommand{\href}[2]{#2}

\end{document}